\documentclass[a4paper,12pt]{amsart}

\usepackage{amsmath}
\usepackage{amsthm}
\usepackage{amssymb}
\usepackage{amsfonts}
\usepackage{enumerate}
\usepackage{mathrsfs}
\usepackage[shortlabels]{enumitem}
\usepackage[dvipsnames]{xcolor}
\usepackage[normalem]{ulem}
\usepackage{marginnote}
\usepackage{esint}
\usepackage{hyperref}
\usepackage{pdfpages}
\usepackage{tikz-cd}
\usepackage{float}

\usetikzlibrary{calc}

\usepackage[alphabetic]{amsrefs}

\calclayout

\theoremstyle{plain}
\newtheorem{theorem}{Theorem}[section]

\newtheorem{lemma}[theorem]{Lemma}
\newtheorem{corollary}[theorem]{Corollary}
\newtheorem{proposition}[theorem]{Proposition}
\newtheorem{conjecture}{Conjecture}

\theoremstyle{definition}
\newtheorem*{definition*}{Definition}
\newtheorem{definition}[theorem]{Definition}
\newtheorem{example}[theorem]{Example}

\theoremstyle{remark}
\newtheorem{remark}[theorem]{Remark}

\numberwithin{equation}{section}

\def \one {\mathsf{1}}

\def \D {\mathcal{D}}

\def \M {\mathcal{M}}
\def \N {\mathbb{N}}

\def \Z {\mathbb{Z}}

\def \eps {\varepsilon}

\def \sbs {\subseteq}
\def \sps {\supseteq}

\DeclareMathOperator{\diam}{diam}

\DeclareMathOperator{\LF}{LF}

\DeclareMathOperator{\supp}{supp}

\DeclareMathOperator{\Hyp}{Hyp}
\DeclareMathOperator{\Dyd}{Dyd}

\begin{document}
\title{Uniformly Lipschitz Hyperbolic Group Actions on $\ell^p$}

\author{Chris Gartland}
\address{Department of Mathematics and Statistics, University of North Carolina at Charlotte, 9201 University City Blvd. Charlotte, NC 28223, USA.}

\author{Tianyi Zheng}
\address{Department of Mathematics, University of California San Diego, 9500 Gilman Dr. La Jolla, CA 92093, USA.}

\thanks{The first named author was supported by the National Science Foundation under Grant Number DMS-2555144. The second named author was supported by the National Science Foundation under Grant Number DMS-2348143}
	


\begin{abstract}
Let $\Gamma$ be a finitely generated hyperbolic group, and let $Q$ denote the conformal dimension of its Gromov boundary $\partial\Gamma$. We prove that for every $p < \frac{Q}{Q-1}$, the group $\Gamma$ admits a proper uniformly Lipschitz affine action on $\ell^p$ with compression exponent $1/p$. The $\ell^p$-space hosting the action is designed from a Markov chain on a specialized hyperbolic filling of $\partial\Gamma$.
\end{abstract}

\maketitle
\tableofcontents

\section{Introduction and Main Result}
\label{sec:intro}

Group actions on Banach spaces play a central role in geometric and analytic group theory. For $p\in [1,2]$, proper isometric affine actions on $L^p$-spaces characterize the Haagerup property (also known as a-T-menability) \cite[Theorem~41]{Now15} -- an analytic property of groups with important connections, for example, to weak amenability and the Baum-Connes conjecture \cite[\S1.3]{CCJJV}. It is well-known that some hyperbolic groups possess the Haagerup property, while others have Kazhdan’s Property (T) (e.g., \cite[\S3.3]{CCJJV}). It is a conjecture of Shalom that every hyperbolic group satisfies a weaker version of the Haagerup property (\cite[Open~Problem~14]{Ober}, \cite[Conjecture~35]{Now15}).

\begin{conjecture}[Shalom] \label{conj:Shalom}
Every finitely generated hyperbolic group admits a proper uniformly Lipschitz\footnote{A group action on a metric space is {\it uniformly Lipschitz} if there exists $L<\infty$ such that the map $x\mapsto \gamma\cdot x$ is $L$-Lipschitz for every group element $\gamma$.} affine action on a Hilbert space.
\end{conjecture}

\noindent We summarize some of the known results related to Conjecture~\ref{conj:Shalom}:
\begin{enumerate}
    \item Yu proved that every hyperbolic group admits a proper {\it isometric} affine action on $\ell^p$ for $p$ {\it sufficiently large} \cite{Yu} (a subsequent presentation was provided in \cite{AL}).
    
    \item Nica \cite[Theorem~22]{Nica} refined Yu's result and proved that every hyperbolic group admits a proper {\it isometric} affine action on an $L^p$-space for $p$ {\it larger} than its hyperbolic dimension. Bourdon \cite{Bourdon} refined this further and showed that $p$ can be taken {\it larger} than the Ahlfors regular conformal dimension (see Remark~\ref{rmk:confdim}) of the Gromov boundary of the group (see the last paragraph of \cite[\S1]{Nica}).
    
    \item Proper {\it uniformly Lipschitz} affine actions on $\ell^1$ have been established for every hyperbolic group (\cite{DM,Gartland}, see also \cite{Vergara} for actions on subspaces of $L^1([0,1])$).
    
    \item It was proved in \cite{Nishikawa} that Sp$(n,1)$ (and its lattices) satisfies Shalom's conjecture. Thus far, these are the only known hyperbolic groups with Property (T) satisfying the conjecture.
\end{enumerate}

We remark that, due to the fact that some hyperbolic groups have Property (T) (see \cite[\S3.6]{Now15}), in items (1) and (2) above, the exponent $p$ cannot be taken to be 2 or smaller for every hyperbolic group, and in item (3), the action cannot be taken to be isometric.

One will notice that, apart from the group Sp$(n,1)$, the works discussed in the items above say nothing about actions of Property (T) hyperbolic groups on $L^p$-spaces for $1 < p \leq 2$. In this article, we establish a result of this type for all hyperbolic groups. See \S\ref{ss:hyperbolic} for the definition of Assouad conformal dimension.

\begin{theorem} \label{thm:main}
Let $\Gamma$ be a finitely generated hyperbolic group. Let $Q$ denote the Assouad conformal dimension of $\partial\Gamma$. Then for all $p < \frac{Q}{Q-1}$, the group $\Gamma$ admits a proper uniformly Lipschitz affine action on $\ell^p$ with compression exponent $1/p$.
\end{theorem}

\noindent Theorem~\ref{thm:main} is restated as Theorem~\ref{thm:main2} and proved at the end of \S\ref{sec:lpactions}. We remark that, just as before, the action in Theorem~\ref{thm:main} cannot be taken isometric due to the existence of hyperbolic groups with Property (T).

We would also like to point the reader to other recent works \cite{BSob,BG,BSlog,GVZ} on actions of hyperbolic groups on $L^p$-spaces motivated (at least in part) by Conjecture~\ref{conj:Shalom}.

\subsection{Proof strategy and its inspiration}
We present in this subsection a rough overview of the proof strategy for Theorem~\ref{thm:main} and provide an account of its inspiration.

\subsubsection{Strategy}
\label{sss:strategy}
Given a group $G$, vector space $V$, and linear representation $\pi$ of $G$ on $V$, we recall that a {\it cocycle} of $\pi$ is a map $b: G \to V$ satisfying $b(gh) = \pi(g)b(h) + b(g)$ for all $g,h\in G$. In this case, the assignment $\alpha(g)v := \pi(g)v + b(g)$ defines an affine action of $G$ on $V$. In case $V$ is normed, the affine action $\alpha$ is uniformly Lipschitz if and only if $\pi$ is uniformly bounded (meaning $\sup_{g\in G}\|\pi(g)\|<\infty$), and if additionally $G$ is finitely generated, then the affine action $\alpha$ is proper if and only if $b$ is proper (meaning $\|b(g)\| \to \infty$ as the word length of $g$ goes to $\infty$).

Let $\Gamma$ be a finitely generated hyperbolic group. We produce, for each $p < \frac{Q}{Q-1}$ where $Q$ is the Assouad conformal dimension of $\partial\Gamma$, a linear representation $\pi$ of $\Gamma$ on an $\ell^p$-space together with a proper cocycle $b$. The coycle $b$ is canonical and its properness is immediate to verify. The novelty is in the construction of the representation $\pi$, and the most difficult part of the proof is in establishing its uniform boundedness. We describe the construction of $\pi$ and $b$ in the remainder of this subsection.

We begin by defining an auxiliary bounded degree graph $\tilde{X}$ that is quasi-isometric to $\Gamma$. This graph comes from a particular type of hyperbolic filling $X$ of $\partial\Gamma$ that we call a {\it (sub)Euclidean dyadic hyperbolic filling} (see Definitions~\ref{def:Euclidean} and \ref{def:subEuclidean} and Figure~\ref{fig:Hyp(Z)}), together with a {\it $K$-leaf extension} $\tilde{X}$ of $X$ (Definition~\ref{def:metricextension}), where we attach to each point $x\in X$ a set of $K$ extra points at unit distance from $x$. Importantly, we choose the filling $X$ so that its exponential base-2 growth rate is as small as possible, given the conformal dimension of $\partial\Gamma$. The design of the leaf extension ensures that $\Gamma$ is not only quasi-isometric to $\tilde{X}$, but is in fact biLipschitz equivalent to a coarsely dense subset of $\tilde{X}$. The graph is equipped with a distinguished basepoint $*\in\tilde{X}$.

The specific combinatorial structure of $\tilde{X}$ permits a naturally defined Markov operator $P$ and associated Laplace operator $\Delta=I-P$. The Markov operator $P$ is designed to satisfy 5 specific properties (Theorem~\ref{thm:admissible}), and we interact with $P$ only through these 5 properties. Geometrically, the Markov operator $P$ acts by moving points $x \neq *$ towards the basepoint $*$, with each application of $P$ decreasing the distance of $x$ to $*$ by a bounded, positive amount. The basepoint $*$ is a fixed point of $P$. Thus, the forward orbit of $x$ under $P$ is a convex combination of quasi-geodesics from $x$ to $*$.

The Laplace operator $\Delta$ is a linear isomorphism from the vector space $\M_{\tilde{X}}(\tilde{X})$ of finitely supported signed measures $\mu$ on $\tilde{X}$ with $\mu(\tilde{X}) = 0$ to the vector space $\M_{*}(\tilde{X})$ of finitely supported signed measures $\nu$ on $\tilde{X}$ with $\nu(\{*\}) = 0$. This fact is proved by constructing the inverse $\Delta^{-1} := \sum_{n=0}^\infty P^n$, where the sum is pointwise-finite by the 5 properties of $P$. By identifying $\M_{*}(\tilde{X})$ with finitely supported functions on $\tilde{X}\setminus\{*\}$ in the canonical way, we can equip the former space with a norm making it isometric to a dense subset of $\ell^p(\tilde{X}\setminus\{*\})$. We denote this normed space by $\M_{*}^p(\tilde{X})$, and we will define the proper uniformly Lipschitz affine action of $\Gamma$ on this normed space.

To describe the action of $\Gamma$ on $\M_{*}^p(\tilde{X})$, assume in this paragraph that $\Gamma$ is actually biLipschitz equivalent to $\tilde{X}$, and hence that it has a uniformly Lipschitz action on $\tilde{X}$. This induces an action $(\gamma,\mu) \mapsto \gamma\cdot\mu$ of $\Gamma$ on $\M_{\tilde{X}}(\tilde{X})$. Then we define a linear representation $\pi$ of $\Gamma$ on $\M_{*}^p(\tilde{X})$ by $\pi(\gamma)\nu := \Delta^{-1}(\gamma\cdot(\Delta \nu))$. We show, through the key central Lemma~\ref{lem:Lipschitz->bndd}, that when $p < \frac{Q}{Q-1}$ where $Q$ is the Assouad conformal dimension of $\partial\Gamma$, the graph $\tilde{X}$ can be chosen so that the resulting representation $\pi$ is uniformly bounded on $\M_{*}^p(\tilde{X})$. The cocycle $b:\Gamma \to \M_{*}^p(\tilde{X})$ for the representation is the natural one $b(\gamma) := \Delta^{-1}(\gamma\cdot\delta_{*}-\delta_{*})$. The cocycle is proper with compression exponent $1/p$ (proof of Theorem~\ref{thm:Markov->lp}).

Now, in the general case, $\Gamma$ is not biLipschitz equivalent to $\tilde{X}$, but only to a coarsely dense subset $R \sbs \tilde{X}$. In this setting, the action of the previous paragraph naturally lives on $\Delta^{-1}\M_{\tilde{X}}(R)$, where $\M_{\tilde{X}}(R) \sbs \M_{\tilde{X}}(\tilde{X})$ is the subspace of measures supported on $R$. Coarse density implies that $R$ is a Lipschitz retract of $\tilde{X}$, and this in turn implies that $\Delta^{-1}\M_{\tilde{X}}(R)$ is isomorphic to a complemented subspace of $\M_*^p(\tilde{X})$ (Lemma~\ref{lem:Lipschitz->bndd}). Therefore, the action on $\Delta^{-1}\M_{\tilde{X}}(R)$ extends to one on all of $\M_*^p(\tilde{X})$ by complementation.

\subsubsection{Inspiration}
The strategy described above was inspired by \cite[Lemma~7.1]{Gartland}, which asserts that the Lipschitz free space $\LF(\Gamma)$ (see \cite[\S2.3]{Gartland} for the definition) is isomorphic to $\ell^1$. By functoriality, $\LF(\Gamma)$ carries a proper isometric affine action of $\Gamma$, and hence this lemma yields a corresponding proper uniformly Lipschitz affine action of $\Gamma$ on $\ell^1$ (\cite[Theorem~7.2]{Gartland}). In \cite{GVZ}, a much more general result was obtained using the same method: every finitely generated group with Property A (see \cite[\S1 and \S2.4]{GVZ} for the definition), including higher rank lattices such as SL$(n,\Z)$ for $n\geq 3$, admits a proper left invariant metric $d$ such that $\LF(\Gamma,d)$ is isomorphic to $\ell^1$, and hence $\Gamma$ admits a proper uniformly Lipschitz affine action on $\ell^1$ by functoriality.

To prove Theorem~\ref{thm:main}, we sought to adapt the proof of \cite[Theorem~7.2]{Gartland} to produce actions of hyperbolic groups $\Gamma$ on subspaces of $\ell^p$ for $p>1$. Towards that end, we came to an understanding that the isomorphism of $\LF(\Gamma)$ with $\ell^1$ obtained in \cite[Lemma~7.1]{Gartland} could be realized in a different way, using Markov chains as the key conceptual tool instead of the stochastic-embeddings-into-trees framework of \cite{Gartland}. This eventually led to the definition of the action on $\Delta^{-1}\M_{\tilde{X}}(R) \sbs \M_*^p(\tilde{X})$ described in \S\ref{sss:strategy}. In the case $p=1$, we believe that $\Delta^{-1}\M_{\tilde{X}}(R)$ is in fact isomorphic to $\LF(\Gamma)$, and that the action of $\Gamma$ on $\Delta^{-1}\M_{\tilde{X}}(R) \sbs \M_*^1(\tilde{X})$ is conjugate to the action of $\Gamma$ on $\LF(\Gamma)$ by this isomorphism, although we do not prove this. Assuming that this belief is true, the uniform Lipschitzness of the action on $\Delta^{-1}\M_{\tilde{X}}(R) \sbs \M_*^1(\tilde{X})$ is then automatic since it is conjugate by a Lipschitz map to the isometric action on $\LF(\Gamma)$. The difficulty that we must overcome to establish uniform Lipschitzness of the action for $p>1$ is that, for general groups $\Gamma$, there is no intermediate ``functorial normed space" $V(\Gamma)$ that is isomorphic to $\Delta^{-1}\M_{\tilde{X}}(R) \sbs \M_*^p(\tilde{X})$ and for which $\Gamma$ acts isometrically. Indeed, if there were such a space $V(\Gamma)$, it would be superreflexive, and thus every finitely generated group $\Gamma$ would act properly and isometrically on a superreflexive space, but this fails even for some Property A groups like higher rank lattices by \cite{Opp} and \cite{dLdlS}. We overcome this difficulty in the technical Lemma~\ref{lem:Lipschitz->bndd} by proving uniform Lipschitzness of the action on $\Delta^{-1}\M_{\tilde{X}}(R) \sbs \M_*^p(\tilde{X})$ directly, without the aid of a well-understood intermediate space. As it turns out, the conformal dimension of $\partial\Gamma$ constrains the values of $p$ for which we can prove uniform Lipschitzness.

\subsection{Organization}
The reader will notice that the conformal dimension of the boundary plays a significant role in both our Theorem~\ref{thm:main} and of the aforementioned work of Bourdon \cite{Bourdon}. In \S\ref{ss:Bourdon}, we describe that work a bit more precisely and compare it to our methods.

In \S\ref{sec:prelims}, we recall fundamental definitions of metric and hyperbolic geometry.

In \S\ref{sec:Euclidean}, we define the (sub)Euclidean dyadic hyperbolic fillings $\Hyp(Z)$ of compact subsets $Z \sbs [0,1]^m$, and establish some of the basic properties. The most important fact is that $\Hyp(\partial\Gamma)$ is quasi-isometric to $\Gamma$ when $\Gamma$ is a finitely generated hyperbolic group. This is stated as Theorem~\ref{thm:qifilling}, but the proof of this theorem is deferred to the Appendix (Theorem~\ref{thm:qifillingA}).

In \S\ref{sec:Markov}, we study general Markov operators satisfying a certain set of 5 properties. We call a Markov operator satisfying the first 3 of these properties {\it admissible} (Definition~\ref{def:admissible}). The other 2 properties are {\it Wasserstein $\frac{1}{2}$-Contractivity} and {\it Dyadic Growth Rate $\leq r$}. We show in Proposition~\ref{prop:admissible} that the filling $\Hyp(Z)$ naturally carries a Markov operator satisfying these 5 properties, with dyadic growth rate $r$ bounded by the Assouad conformal dimension of $Z$. A Markov operator with the same 5 properties is then defined on the $K$-leaf extension of $\Hyp(Z)$ (Proposition~\ref{prop:graphextension}).

In \S\ref{sec:lpactions}, we use the machinery built in \S\ref{sec:Euclidean} and \S\ref{sec:Markov} to construct proper uniformly Lipschitz affine actions of finitely generated hyperbolic groups on $\ell^p$ for $p < \frac{Q}{Q-1}$, where $Q$ is the Assouad conformal dimension of $\partial\Gamma$ (Theorem~\ref{thm:main2}). See \S\ref{sss:strategy} for an overview of the construction of this action.

\subsection{Comparison with the work of Bourdon}
\label{ss:Bourdon}
In \cite{Bourdon}, Bourdon proved that when $\Gamma$ is a finitely generated hyperbolic group, a direct sum of finitely many copies of the regular representation of $\Gamma$ on $\ell^p(\Gamma)$ admits a proper cocycle whenever $p > Q$ and $Q$ is the Ahlfors regular conformal dimension of $\partial\Gamma$. The cocycle is constructed from elements of $\ell^p$-cohomology of $\Gamma$, which in turn is identified with a Besov space on $\partial\Gamma$ that is nontrivial when $p>Q$. This identification is due to earlier work of Bourdon-Pajot \cite{BP}, through the use of hyperbolic fillings.

We also use hyperbolic fillings in an essential way in the present work. The main difference is that we use fillings to construct the linear representation and our cocycle is canonical, whereas in \cite{Bourdon}, the linear representation is canonical and fillings are used to construct the cocycle. The other obvious differences are that the representations of \cite{Bourdon} are isometric and on $\ell^p$ for $p > Q$, while ours are uniformly bounded and on $\ell^p$ for $p< \frac{Q}{Q-1}$. As already discussed, the fact that our representations are only uniformly bounded instead of isometric is necessary due to the existence of hyperbolic groups with Property (T).

\section{Preliminaries} \label{sec:prelims}
\subsection{Geometric Mappings}
Let $(X,d_X)$ and $(Y,d_Y)$ be metric spaces. A map $f: X \to Y$ is {\it coarse Lipschitz} if there exist $L < \infty$ and $C < \infty$ such that
$$d_Y(f(x), f(y)) \le L d_X(x,y) + C$$
for all $x, y \in X$. If $C$ can be taken to be 0, then the map is {\it Lipschitz}. An injective Lipschitz map with Lipschitz inverse is a {\it biLipschitz embedding}, and a surjective biLipschitz embedding is a {\it biLipschitz equivalence}. 

A subset $A \sbs X$ is {\it coarsely dense} in $X$ if there exists $C < \infty$ such that for all $x \in X$,
$$d_X(x, A) = \inf_{a \in A} d_X(x, a) \le C.$$
Two maps $f: X \to Y$ and $g: Y \to X$ are {\it coarsely inverse} if there exists a constant $C < \infty$ such that
$$d_X(g(f(x)), x) \le C \quad \text{and} \quad d_Y(f(g(y)), y) \le C$$
for all $x \in X$ and all $y \in Y$. A map $f: X \to Y$ is a {\it quasi-isometry} if it is coarse Lipschitz and admits a coarse inverse $g: Y \to X$ that is also coarse Lipschitz.

The following fact is a simple consequence of Zorn's Lemma and we leave the proof to the reader.

\begin{lemma} \label{lem:qi->biLip1}
Let $f: X \to Y$ be a quasi-isometry between metric spaces. Then there exist coarsely dense subsets $N_X \sbs X$, $N_Y \sbs Y$ such that the restriction of $f$ to $N_X$ is biLipschitz with $f(N_X) = N_Y$.
\end{lemma}

We say that $X$ has {\it sharp Assouad dimension $\leq s$} (for some $s\in[0,\infty)$) if there exists a constant $C < \infty$ such that for every radius $0<r<R<\infty$ and center $x\in X$, the ball $B_R(x) \sbs X$ can be covered by at most $C(R/r)^s$ balls of radius $r$. For $\alpha \in (0,1]$, the {\it $\alpha$-snowflake} of $(X,d_X)$ is the metric space $(X, d_X^\alpha)$. For any $\alpha\in(0,1]$, it is clear that $X$ has sharp Assouad dimension $\leq s$ if and only if its $\alpha$-snowflake has sharp Assouad dimension $\leq s/\alpha$. We say that $X$ is {\it doubling} if it has sharp Assouad dimension $\leq s$ for some $s<\infty$. The {\it Assouad embedding theorem} states that if $\alpha<1$, then the $\alpha$-snowflake of any doubling space admits a biLipschitz embedding into a finite-dimensional normed space (\cite{Assouad}, see \cite[Theorem~12.2]{Heinonen}).

\begin{definition}[Metric $K$-Leaf Extension]
\label{def:metricextension}
Let $(X,d_X)$ be a metric space and $K\in\N$. The {\it metric $K$-leaf extension} of $X$ is a new metric space $(\tilde{X},d_{\tilde{X}})$ formed by attaching $K$ leaves to each point in $X$ at unit distance. Specifically, the underlying set is $\tilde{X} := X \sqcup (\{1,2,\dots K\} \times X)$, and the metric is defined by
\begin{equation*}
    d_{\tilde{X}}(\tilde{x},\tilde{y}) = \begin{cases}
        d(x,y) & \tilde{x} = x,\tilde{y} = y \in X \\
        d(x,y)+1 & \tilde{x} = (i,x) \in \tilde{X}\setminus X \, , \, \tilde{y} = y \in X \\
        d(x,y)+2 & \tilde{x} = (i,x), \tilde{y} = (j,y) \in \tilde{X}\setminus X
    \end{cases}.
\end{equation*}
The set of points $\{1,2,\dots K\} \times \{x\} \sbs \tilde{X}$ are called the {\it leaves} of $x \in X$.
\end{definition}

Recall that a metric space $(X,d_X)$ is {\it uniformly discrete} if $\inf_{x\neq y\in X} d_X(x,y) > 0$ and {\it uniformly locally finite} if $\sup_{x\in X}|B_R(x)|<\infty$ for every $R<\infty$. The important examples of uniformly discrete, uniformly locally finite metric spaces we concern ourselves with in this article are the vertex sets of connected, bounded degree graphs, equipped with the (unweighted) shortest path metric. In this setting, we can use Lemma~\ref{lem:qi->biLip1} to upgrade a quasi-isometric map to a biLipschitz one.

\begin{lemma} \label{lem:qi->biLip2}
Let $X$ be a uniformly discrete, uniformly locally finite metric space and $Y$ a metric space. If $X$ is quasi-isometric to $Y$, then there exists a constant $K<\infty$ such that $X$ is biLipschitz equivalent to a coarsely dense subset of the metric $K$-leaf extension $\tilde{Y}$ of $Y$ of Definition~\ref{def:metricextension}.
\end{lemma}

\begin{proof}
Assume that $X$ is quasi-isometric to $Y$. Then by Lemma~\ref{lem:qi->biLip1}, there exist coarsely dense subsets $N_X \sbs X$, $N_Y \sbs Y$ and a biLipschitz equivalence $f: N_X \to N_Y$. Let $\pi: X \to N_X$ be a nearest neighbor projection. Then $\pi$ is Lipschitz since $X$ is uniformly discrete and $N_X$ is coarsely dense. Furthermore, since $X$ is uniformly locally finite, there is a constant $K\in\N$ such that the fiber $\pi^{-1}(\{x\})$ has cardinality at most $K$ for all $x\in N_X$. Choose for each $x \in N_X$ an injection $\iota_x: \pi^{-1}(\{x\}) \to \{1,2,\dots K\}$. Then we map $X$ into $\tilde{Y}$ by sending each $x$ to a leaf of $f(\pi(x))$. Specifically, the map is defined by $x \mapsto (\iota_{\pi(x)}(x),f(\pi(x)))$. It is straightforward to check that this is a biLipschitz embedding such that the 1-neighborhood of its image contains $N_Y$, and hence is coarsely dense in $\tilde{Y}$.
\end{proof}

\subsection{Hyperbolic metric spaces}
\label{ss:hyperbolic}
We mainly follow the presentation of hyperbolic metric spaces and their boundaries as given in Bonk-Schramm \cite{BS} while also citing the more comprehensive accounts \cite{Buyalo} and \cite[Chapter~III.H]{BH} (especially in Appendix~\ref{app}).

Let $(X,d_X)$ be a metric space and $x,y,o \in X$. The \emph{Gromov product} of $x,y$ with respect to $o$ is defined by
\begin{equation*} \label{eq:Gromovproductdef}
    (x|y)_o := \frac{1}{2}(d_X(x,o)+d_X(y,o)-d_X(x,y)).
\end{equation*}
Let $\delta \in [0,\infty)$. Then $X$ is said to be \emph{$\delta$-hyperbolic} if for all $x,y,z,o \in X$,
\begin{equation*}
    (x|z)_o \geq \min\{(x|y)_o,(y|z)_o\} - \delta.
\end{equation*}
The space $X$ is \emph{Gromov hyperbolic}, or just \emph{hyperbolic}, if it is $\delta$-hyperbolic for some $\delta \in [0,\infty)$.

We say that a finitely generated group is {\it hyperbolic} is its Cayley graph with respect to some finite generating set is a hyperbolic metric space. As is well-known, hyperbolicity is a quasi-isometric invariant among geodesic metric spaces \cite[Corollary~1.3.5]{Buyalo}, and thus the property of a group being hyperbolic is independent of the choice of generating set (since any two finite generating sets yield quasi-isometric Cayley graphs).

Suppose $(X,d_X)$ is a hyperbolic space. A sequence $\{x_i\}_{i=1}^\infty \sbs X$ is said to \emph{converge at $\infty$} if for some $o \in X$, it holds that $\lim_{i,j\to\infty} (x_i|x_j)_o = \infty$. Two sequences $\{x_i\}_{i=1}^\infty, \{y_i\}_{i=1}^\infty \sbs X$ converging at $\infty$ are \emph{equivalent} if for some $o \in X$, it holds that $\lim_{i\to\infty} (x_i|y_i)_o = \infty$.

This is an equivalence relation on the set of sequences converging at $\infty$ by hyperbolicity, and the set of equivalence classes is called the \emph{Gromov boundary}, or just \emph{boundary}, of $X$, denoted by $\partial X$. The \emph{Gromov product} of $\xi,\upsilon \in \partial X$ with respect to $o \in X$ is defined by
\begin{equation*}
    (\xi|\upsilon)_o := \sup \{\liminf_{i\to\infty} (x_i|y_i)_o: \{x_i\}_{i=1}^\infty \in \xi,\{y_i\}_{i=1}^\infty \in \upsilon\}.
\end{equation*}
A metric $d$ on $\partial X$ is said to be \emph{visual} if there exist $o \in X$, $\eps > 0$, and $C< \infty$ such that $C^{-1}e^{-\eps(\xi|\upsilon)_o} \leq d(\xi,\upsilon) \leq Ce^{-\eps(\xi|\upsilon)_o}$ for all $\xi,\upsilon \in \partial X$ (in Bonk-Schramm, the set of visual metrics is called the \emph{canonical $B$-structure} on $\partial X$ \cite[page~282]{BS}). Visual metrics on $\partial X$ always exist and are unique up to snowflake equivalence, and it is clear from the definition that any snowflake of a visual metric is visual and any metric biLipschitz equivalent to a visual metric is visual. Boundaries are always complete and bounded.

An injective map $f: Y \to Z$ between metric spaces is a {\it snowflake embedding} if for all $x, y \in Y$,
$$\lambda^{-1} d_Y(x,y)^\alpha \le d_Z(f(x),f(y)) \le \lambda d_Y(x,y)^\alpha,$$
for some $\alpha\in(0,\infty)$ and $\lambda<\infty$ and is a {\it power quasisymmetric embedding} if for all distinct triples $x, y, z \in Y$,
$$\frac{d_Z(f(x),f(z))}{d_Z(f(x),f(y))} \le \eta_{\alpha,\lambda}\left( \frac{d_Y(x,z)}{d_Y(x,y)} \right)$$
for some $\alpha\in(0,\infty)$ and $\lambda<\infty$, where
$$\eta_{\alpha,\lambda}(t) = 
\begin{cases} 
\lambda t^{1/\alpha} & 0 < t < 1 \\[4pt]
\lambda t^\alpha & 1 \le t
\end{cases}.$$
Surjective snowflake embeddings and power quasisymmetric embeddings are called {\it snowflake equivalences} and {\it power quasisymmetric equivalences}, respectively. It is clear that a biLipschitz embedding/equivalence is a snowflake embedding/equivalence, which in turn is a power quasisymmetric embedding/equivalence. Thus, since any two visual metrics on the boundary of a hyperbolic metric space are snowflake equivalent, the statement ``$\partial X'$ is power quasisymmetrically equivalent to $\partial X$'' is well-defined independent of the choice of visual metric.

The {\it Assouad conformal dimension} of $\partial X$ is the infimum over all $s\geq 0$ such that $\partial X$ is power quasisymmetrically equivalent to a space with sharp Assouad dimension $\leq s$.

\begin{remark}[Conformal Dimensions]
\label{rmk:confdim}
There are other notions of conformal dimension common in the literature that end up being equivalent to the Assouad conformal dimension for boundaries of hyperbolic groups (see \cite{SEB} and references therein). First of all, we note that the typical definition of Assouad conformal dimension is the infimum of sharp Assouad dimensions over all metrics {\it quasisymmetric} to a visual metric, instead of power quasisymmetric. Quasisymmetries are more general than power quasisymmetries, but the two notions coincide for homeomorphisms between uniformly perfect spaces \cite[Corollary~3.12]{TV}, and boundaries of hyperbolic groups are uniformly perfect (this follows from their local self-similarity \cite[Theorem~2.3.2]{Buyalo}). One can define the {\it Ahlfors regular conformal dimension} of a metric space $(Z,d)$ as the infimum of sharp Assouad dimensions over all {\it Ahlfors regular} metrics on $Z$ quasisymmetric to $d$. However, the two notions coincide on uniformly perfect spaces \cite[Proposition~2.2.6, \S7.1]{Tyson}.
\end{remark}

\section{Euclidean Dyadic Hyperbolic Fillings}
\label{sec:Euclidean}

In order to construct the Laplace operator used in the proof of our main result (Theorem~\ref{thm:main}), we need to employ a quasi-isometric model of our hyperbolic group with specific combinatorial structure. It turns that hyperbolic fillings of Euclidean cubes and their subsets provide structure. We we will use fillings constructed from dyadic scales, but any other fixed geometric scale would work equally well, up to the implicit constants.

\begin{definition}[Euclidean Dyadic Hyperbolic Fillings] \label{def:Euclidean}
Consider the sequence of dyadic fractions $\Dyd_0 := \{0,1\} \sbs \Dyd_1 := \{0,\frac12,1\} \sbs \Dyd_2 := \{0,\frac14,\frac12,\frac34,1\} \sbs \dots [0,1]$, given, in general, by the formula $\Dyd_n := \{j2^{-n}: j \in \Z\cap[0,2^n]\}$ for $n\geq 0$. Each set $\Dyd_n$ naturally carries the structure of a directed path graph, where $x\to y$ is a directed edge if and only if $y-x=2^{-n}$. Here and henceforth, we use the notation ``$x \to y$" to denote the ordered pair $(x,y)$ whenever $(x,y)$ is a directed edge in a directed graph.

Let $m\in\N$ and $n\geq 0$. We consider the $m$-dimension grid graph defined by $\Dyd_n^m := \Dyd_n \square \Dyd_n \square \dots \Dyd_n$. The maximal degree of a vertex in $\Dyd_n^m$ is $2m$. Form the total graph $\Dyd^m$ by vertically gluing together the grid graphs $\Dyd^m_n$ (thought of as horizontal levels) and a basepoint $\{*\}$ on top. Specifically,
$$\Dyd^m := \sqcup_{n\geq -1} \Dyd_n^m := \cup_{n\geq -1} \{n\}\times \Dyd_n^m,$$
where we interpret $\{-1\}\times \Dyd_{-1}^m$ as $\{*\}$. We will use the notation $\{n\} \times u$ to denote elements of $\{n\}\times\Dyd_n^m$, in order to more clearly distinguish the first coordinate of the vertex as indicating its level.

The {\it horizontal edge set $E_h$} of $\Dyd^m$ is defined to be the set of edges internal to each grid graph $\{n\}~\times~\Dyd_n^m$, which we identify with $\Dyd_n^m$. That is, the horizontal edges are of the form $(\{n\} \times u) \to (\{n\} \times v)$ where $u \to v$ is a directed edge in $\Dyd_n^m$.

For $n\geq 1$, we connect each vertex $\{n-1\}\times u \in \{n-1\}\times \Dyd_{n-1}^m$ with its corresponding vertex $\{n\}\times u \in \{n\}\times \Dyd_n^m$ sitting directly below it, with direction $(~\{~n~\}~\times~u~)~\to~(~\{~n~-~1~\}~\times~u~)$. The {\it vertical edge set} $E_v$ consists of these directed edges together with all directed edges of the form $x\to *$, where $x \in \{0\}\times \Dyd_0^m$ and $*$ is the basepoint. The degree of $*$ in the graph $\Dyd^m$ is $2^m$, and all other vertices have degree at most $2m+2$. See Figure~\ref{fig:Hyp(Z)} for an illustration of $\Dyd^2$.
\end{definition}

\begin{definition}[SubEuclidean Dyadic Hyperbolic Fillings] \label{def:subEuclidean}
Let $m\in\N$ and $Z \sbs [0,1]^m$ be nonempty and compact. For each $n\geq 0$, let $\D_n(Z)$ denote the collection of $n$th dyadic cubes which intersect $Z$. Specifically, $Q \in \D_n(Z)$ if $Q = [x_1,y_1] \times [x_2,y_2] \times \dots [x_m,y_m]$, where each $x_i \to y_i$ is an edge of $\Dyd_{n}$, and $Q \cap Z \neq \emptyset$. Given $Q \in \D_n(Z)$, let $Q^{(0)}$ denote its 0-skeleton. That is, $Q^{(0)}$ is the set of vertices of the form $(z_1,\dots z_m) \in \Dyd^m_n$ where $z_i \in \{x_i,y_i\}$ and $Q = [x_1,y_1] \times \dots [x_m,y_m]$. We let $\Hyp(Z)$ be the induced subgraph of $\Dyd^m$ with vertex set $\cup_{n\geq 0} \cup_{Q \in \D_n(Z)} \{n\} \times Q^{(0)} \cup \{*\}$, equipped with intrinsic (unweighted) shortest path metric $d_{\Hyp(Z)}$. We emphasize here that the union is not disjoint; a point $u\in[0,1]^m$ can belong to $Q^{(0)}$ for multiple $Q\in \D_n(Z)$ for a fixed $n\geq 0$. We call $\Hyp(Z)$ the {\it (Euclidean dyadic) hyperbolic filling} of $Z$. See Figure~\ref{fig:Hyp(Z)} for an illustration of the hyperbolic filling of the 4-corner Cantor subset\footnote{This subset is formed by starting with the square $[0,1]^2$, keeping its 4 corner subsquares of side length $1/4$ and throwing away the rest of the set, then keeping inside each of those subsquares their respective 4 corner subsquares of side length $1/16$ and throwing away the rest of the set, ad infinitum. See \cite[Fig.~1.2]{Falconer}.} of $[0,1]^2$.
\end{definition}

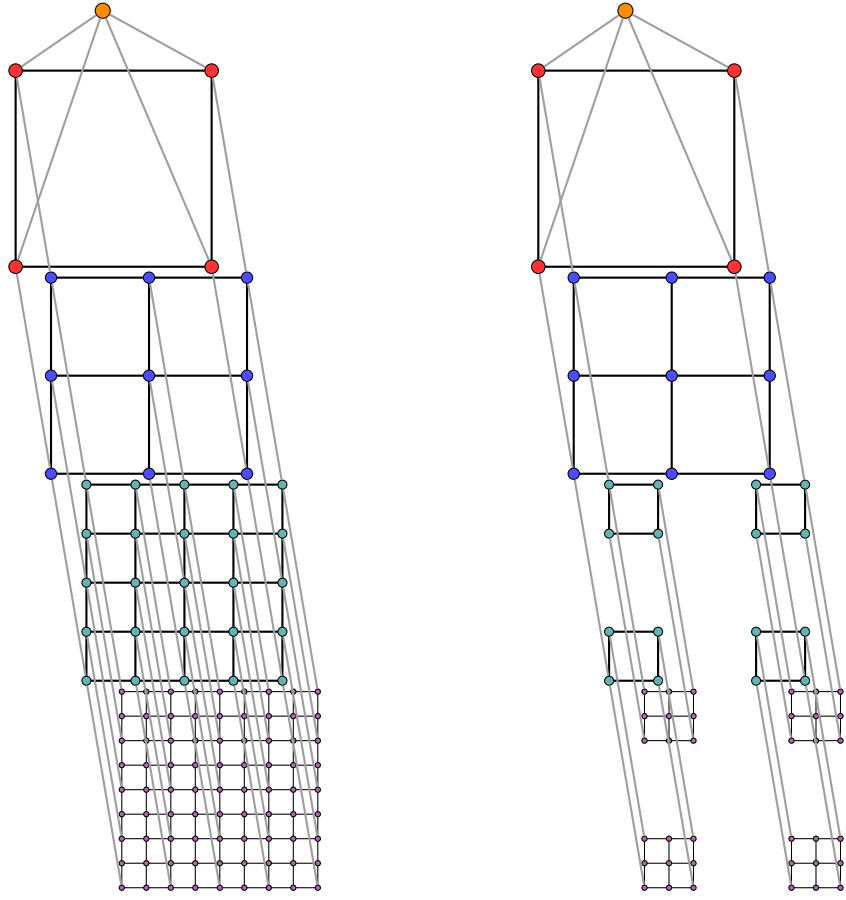
\begin{figure}
\centering

\begin{tikzpicture}[
    scale=.72,
    levelMinus1/.style={
        circle,draw=black,fill=orange!90!yellow,inner sep=2pt
    },
    level0/.style={
        circle,draw=black,fill=red!80,inner sep=1.8pt
    },
    level1/.style={
        circle,draw=black,fill=blue!70,inner sep=1.5pt
    },
    level2/.style={
        circle,draw=black,fill=teal!60,inner sep=1.2pt
    },
    level3/.style={
        circle,draw=black,fill=violet!60,inner sep=.7pt
    },
    verticalEdge/.style={gray!75,thick},
    horizontalEdge/.style={black,thick}
]


\def\Side{3.6}

\def\Slant{0.65}

\def\Yzero{5.2}
\def\Yone{1.4}
\def\Ytwo{-2.4}
\def\Ythree{-6.2}

%

\newcommand{\squaregrid}[6]{%
    \pgfmathsetmacro{\half}{\Side/2}
    \pgfmathsetmacro{\step}{\Side/(#2-1)}

    \foreach \r in {1,...,#2}{
        \foreach \c in {1,...,#2}{
            \pgfmathtruncatemacro{\k}{\c+(\r-1)*#2}
            \pgfmathsetmacro{\x}{#3-\half+(\c-1)*\step}
            \pgfmathsetmacro{\y}{#4+\half-(\r-1)*\step}
            \coordinate (c#1-\k) at (\x,\y);
        }
    }

    \foreach \r in {1,...,#2}{
        \pgfmathtruncatemacro{\a}{1+(\r-1)*#2}
        \pgfmathtruncatemacro{\b}{#2+(\r-1)*#2}
        \draw[horizontalEdge,#6]
            (c#1-\a)--(c#1-\b);
    }

    \foreach \c in {1,...,#2}{
        \pgfmathtruncatemacro{\b}{\c+(#2-1)*#2}
        \draw[horizontalEdge,#6]
            (c#1-\c)--(c#1-\b);
    }

    \pgfmathtruncatemacro{\N}{#2*#2}
    \foreach \k in {1,...,\N}{
        \node (#1-\k) at (c#1-\k) [#5] {};
    }
}


\newcommand{\squarecoords}[4]{%
    \pgfmathsetmacro{\half}{\Side/2}
    \pgfmathsetmacro{\step}{\Side/(#2-1)}

    \foreach \r in {1,...,#2}{
        \foreach \c in {1,...,#2}{
            \pgfmathtruncatemacro{\k}{\c+(\r-1)*#2}
            \pgfmathsetmacro{\x}{#3-\half+(\c-1)*\step}
            \pgfmathsetmacro{\y}{#4+\half-(\r-1)*\step}
            \coordinate (c#1-\k) at (\x,\y);
        }
    }
}

%

\newcommand{\subgrid}[7]{%
    \pgfmathtruncatemacro{\m}{#5-1}

    \foreach \d in {0,...,\m}{
        \pgfmathtruncatemacro{\a}{#4+(#3+\d-1)*#2}
        \pgfmathtruncatemacro{\b}{\a+#5-1}
        \draw[horizontalEdge,#7]
            (c#1-\a)--(c#1-\b);
    }

    \foreach \d in {0,...,\m}{
        \pgfmathtruncatemacro{\a}{#4+\d+(#3-1)*#2}
        \pgfmathtruncatemacro{\b}{\a+(#5-1)*#2}
        \draw[horizontalEdge,#7]
            (c#1-\a)--(c#1-\b);
    }

    \foreach \rr in {0,...,\m}{
        \foreach \cc in {0,...,\m}{
            \pgfmathtruncatemacro{\k}{
                #4+\cc+(#3+\rr-1)*#2
            }
            \node (#1-\k) at (c#1-\k) [#6] {};
        }
    }
}

%

\newcommand{\refine}[6]{%
    \foreach \r in {#5}{
        \foreach \c in {#6}{
            \pgfmathtruncatemacro{\a}{
                \c+(\r-1)*#2
            }
            \pgfmathtruncatemacro{\b}{
                2*\c-1+(2*\r-2)*#4
            }
            \draw[verticalEdge]
                (#1-\a)--(#3-\b);
        }
    }
}


\foreach \S/\panelx/\full in {
    L/-4.8/1,
    R/ 4.8/0
}{
\begin{scope}[xshift=\panelx cm]


    \node (root\S)
        at (-1.4,8.1)
        [levelMinus1] {};

    \pgfmathsetmacro{\Xzero}{-1.2}
    \pgfmathsetmacro{\Xone}{\Xzero+\Slant}
    \pgfmathsetmacro{\Xtwo}{\Xzero+2*\Slant}
    \pgfmathsetmacro{\Xthree}{\Xzero+3*\Slant}


    \squaregrid
        {0\S}{2}
        {\Xzero}{\Yzero}
        {level0}{}

    \foreach \i in {1,...,4}{
        \draw[verticalEdge]
            (root\S)--(0\S-\i);
    }


    \squaregrid
        {1\S}{3}
        {\Xone}{\Yone}
        {level1}{}

    \refine
        {0\S}{2}
        {1\S}{3}
        {1,2}{1,2}


    \ifnum\full=1

        \squaregrid
            {2\S}{5}
            {\Xtwo}{\Ytwo}
            {level2}{}

        \refine
            {1\S}{3}
            {2\S}{5}
            {1,2,3}{1,2,3}

        \squaregrid
            {3\S}{9}
            {\Xthree}{\Ythree}
            {level3}{thin}

        \refine
            {2\S}{5}
            {3\S}{9}
            {1,2,3,4,5}
            {1,2,3,4,5}


    \else


        \squarecoords
            {2\S}{5}
            {\Xtwo}{\Ytwo}

        \foreach \r/\c in {
            1/1,
            1/4,
            4/1,
            4/4
        }{
            \subgrid
                {2\S}{5}
                {\r}{\c}{2}
                {level2}{}
        }

        \refine
            {1\S}{3}
            {2\S}{5}
            {1,3}{1,3}


        \squarecoords
            {3\S}{9}
            {\Xthree}{\Ythree}

        \foreach \r/\c in {
            1/1,
            1/7,
            7/1,
            7/7
        }{
            \subgrid
                {3\S}{9}
                {\r}{\c}{3}
                {level3}{thin}
        }

        \refine
            {2\S}{5}
            {3\S}{9}
            {1,2,4,5}
            {1,2,4,5}

    \fi

\end{scope}
}

\end{tikzpicture}

\caption{
Left: The dyadic hyperbolic filling $\Dyd^2$ of $[0,1]^2$.
Right: The dyadic hyperbolic filling $\Hyp(Z)$ of the
4-corner Cantor set $Z\sbs[0,1]^2$.
}
\label{fig:Hyp(Z)}
\end{figure}

For future use, we make the following observations for a given $n\geq 1$ and cube $Q \in \D_n(Z)$:
\begin{itemize}
    \item There exists a unique cube $\hat{Q} \in \D_{n-1}(Z)$, called the {\it parent} of $Q$, such that $\hat{Q} \sps Q$.
    \item There exists a unique vertex $u_Q \in Q^{(0)} \cap \hat{Q}^{(0)}$.
    \item The ordered pair $(\{n\}\times u_Q) \to (\{n-1\}\times u_Q)$ is a vertical edge of $\Hyp(Z)$.
    \item The diameter with respect to the intrinsic path metric of $\{n\} \times Q^{(0)}$ is $m$. In particular, if $v \in Q^{(0)}$, then $\{n\} \times v$ is connected to $\{n\}\times u_Q$ via a horizontal edge path of $\Hyp(Z)$ of length at most $m$.
\end{itemize}

We will need the following basic but important lemma regarding the geometry of hyperbolic fillings.

\begin{lemma} \label{lem:Hyp(Z)distance}
Let $m\in\N$ and $Z \sbs [0,1]^m$ be nonempty and compact. For every constant $C<\infty$, there exists $C'<\infty$ such that for all $n\geq 0$ and $\eps \in \{0,1\}$, if $\{n\} \times u, \{n+\eps\} \times v \in \Hyp(Z)$ with $\|u-v\| \leq C2^{-n}$, then $d_{\Hyp(Z)}(\{n\} \times u, \{n+\eps\} \times v) \leq C'$, where $\|\cdot\|$ denotes the $\ell^\infty$ norm on $[0,1]^m$.
\end{lemma}

\begin{proof}
Let $C<\infty$, $n\geq 0$, $\eps \in \{0,1\}$, and $\{n\} \times u, \{n+\eps\} \times v \in \Hyp(Z)$ with $\|u-v\| \leq C2^{-n}$. Let $Q_u \in \D_n(Z)$ and $Q_v \in \D_{n+\eps}(Z)$ with $u \in Q_u^{(0)}$ and $v \in Q_v^{(0)}$. It follows from $\|u-v\| \leq C2^{-n}$ that there is a number $k \in \N$, depending only on $C$ (and possibly $m$), and cubes $\tilde{Q}_u,\tilde{Q}_v \in \D_{n-k}(Z)$ such that $u\in \tilde{Q}_u$, $v\in \tilde{Q}_v$, and $\tilde{Q}_u \cap \tilde{Q}_v \neq\emptyset$. From this, it is easy to see, using the observations following Definition~\ref{def:subEuclidean}, that there is a path in $\Hyp(Z)$ connecting $\{n\} \times u$ to $\{n+\eps\} \times v$ of length at most $m(2k+3)+2k+1$. Indeed, if we let $w \in \tilde{Q}_u^{(0)} \cap \tilde{Q}_v^{(0)}$, then $\{n-k\} \times w$ can be reached from $\{n\}\times u$ by a sequence of edges that starts with (at most) $m$ horizontal edges, then alternates, $k$ times, between one vertical edge and (at most) $m$ horizontal edges. The same thing is true for reaching $\{n-k\} \times w$ from $\{n+\eps\}\times v$, except with $k+\eps$ alternations (note that $k+\eps\leq k+1$). This shows that the desired conclusion holds with $C' \leq m(2k+3)+2k+1$.
\end{proof}

It is important for us to quantify the exponential growth rate of the graphs $\Hyp(Z)$, and we do so in the next definition and lemma.

\begin{definition}[Graph Dyadic Growth Rate]
\label{def:dyadicgrowthrate}
Let us say that one vertex $\{n+j\}\times u \in \Hyp(Z)$ {\it lies below} another $\{n\}\times v \in \Hyp(Z)$ if $j > 0$ and there are cubes $Q_u \in \D_{n+j}(Z)$, $Q_v \in \D_n(Z)$ such that $u \in Q_u^{(0)}$, $v \in Q_v^{(0)}$, and $Q_u \sbs Q_v$. Let $r \geq 0$. We say that $\Hyp(Z)$ has {\it dyadic growth rate} $\leq r$ if there exists a constant $C<\infty$ such that for every $\{n\} \times v \in \Hyp(Z)$ and every $j > 0$,
$$|\{\{n+j\}\times u \in \Hyp(Z): \{n+j\}\times u \text{ lies below } \{n\}\times v\}| \leq C2^{rj}.$$
\end{definition}

\begin{lemma} \label{lem:Assouad->dyadicgrowthrate}
Let $m\in\N$ and $Z \sbs [0,1]^m$ be nonempty and compact. Let $r \in [0,m]$. If $Z$ has sharp Assouad dimension $\leq r$, then $\Hyp(Z)$ has dyadic growth rate $\leq r$.
\end{lemma}

\begin{proof}
Assume that $Z$ has sharp Assouad dimension $\leq r$. Let $\{n\}\times v \in \Hyp(Z)$ and $j>0$. Let $Q_v \in \D_n(Z)$ such that $v \in Q_v^{(0)}$. Obviously, to prove the lemma, it suffices to show that
$$|\{Q\in\D_{n+j}(Z): Q \sbs Q_v\}| \leq C2^{rj}$$
for some $C<\infty$ (independent of $j,Q_v$). Since $\diam(Q_v) = 2^{-n}$ and $v\in Q_v$, the triangle inequality implies $Q_v \sbs B_{2^{-n}}(v)$. Here and throughout the proof, the set $B_R(x)$ indicates the ball in $[0,1]^m$, with respect to the $\ell^\infty$ metric, of radius $R$ and center $x$. Hence, it suffices to show that
$$|\{Q\in\D_{n+j}(Z): Q \sbs B_{2^{-n}}(v)\}| \leq C2^{rj}.$$

Since $Z$ has sharp Assouad dimension $\leq r$, there exists a constant $C'\in\N$ (independent of $j,Q_v$) such that $B_{2^{-n}}(v)$ can be covered by $C'2^{rj}$ balls $\{B_{2^{-n-j}}(x_i)\}_{i=1}^{C'2^{rj}}$ of radius $2^{-n-j}$. Hence, it suffices to show that
$$\sum_{i=1}^{C'2^{rj}}|\{Q\in\D_{n+j}(Z): Q \cap B_{2^{-n-j}}(x_i) \neq \emptyset\}| \leq C2^{rj}.$$
It is easy to see that the cardinality inside this summation is bounded (independent of $j,Q_v$), completing the proof.
\end{proof}

The following theorem can be proved using arguments that are now quite standard in the hyperbolic metric geometry literature. Because it is of such importance for our main result (Theorem~\ref{thm:main}), we include a detailed proof in the Appendix (see Theorem~\ref{thm:qifillingA}).

\begin{theorem}[Fundamental Theorem of SubEuclidean Hyperbolic Fillings (Theorem~\ref{thm:qifillingA})] \label{thm:qifilling}
Let $\Gamma$ be a finitely generated hyperbolic group. Let $m\in\N$ and $Z \sbs [0,1]^m$ such that $\partial\Gamma$ is power quasisymmetrically equivalent to $Z$. Then $\Gamma$ is quasi-isometric to the hyperbolic filling $\Hyp(Z)$ of Definition~\ref{def:subEuclidean}.
\end{theorem}

\section{Admissible Markov Operators}
\label{sec:Markov}

In this section we define a class of Markov operators $P$ that we will used to construct actions of hyperbolic groups on $\ell^p$, see \S\ref{sss:strategy} for an informal discussion of the strategy.

Let $X$ be a set. Let $\M_{00}(X)$ denote the vector space of finitely supported signed measures on $X$. For $S\sbs X$, we write $\M_{00}(S)$ to denote the subspace of all $\mu\in\M_{00}(X)$ with $\supp \mu := \{x\in X: \mu(\{x\}) \neq 0\} \sbs S$. Let $P$ be a (locally finite) Markov operator on $X$, by which we mean a linear map $P: \M_{00}(X)\to\M_{00}(X)$ preserving the convex subset of probability measures.

\begin{definition}[Admissible Markov Operators]
\label{def:admissible}
Now assume that $(X,E)$ is a connected, directed, bounded degree graph, equipped with the shortest path metric $d_X$ and a basepoint $* \in X$. Let us say that a Markov operator $P$ on $X$ is {\it admissible} if it possesses the following properties: 
\begin{enumerate}
    \item \fbox{Gradation} There exists a partition $X = \sqcup_{n\geq -1} X_n$, called a {\it grading}, such that $X_{-1} = \{*\}$ is an absorbing state of $P$ and $P(\M_{00}(X_n)) \sbs \M_{00}(X_{n-1})$ for all $n\geq 0$.
    \item \fbox{Morse Property} There exists a constant $C_1<\infty$ such that the cardinality $|\supp P^n(\delta(x))|$ is at most $C_1$ for all $x\in X$ and all $n\geq 0$, where $\delta(x)$ denotes the Dirac measure at $x$.
    \item \fbox{Locality} There exists a constant $C_2<\infty$ such that $d_X(x,y) \leq C_2$ for all $x\in X$ and all $y \in \supp P(\delta(x))$.
\end{enumerate}
There are two additional important properties that an admissible Markov operator may satisfy:
\begin{enumerate}[resume]
    \item \fbox{Wasserstein $\frac{1}{2}$-Contractivity} there exists a partition of $E = E_{v} \sqcup E_{h}$ into {\it vertical} and {\it horizontal edges} such that
    \begin{itemize}
        \item for every vertical edge $(x_v \to y_v)\in E_{v}$, $x_v\neq *$ and $P(\delta(x_v)) = \delta(y_v)$ and
        \item for every horizontal edge $(x_h \to y_h)\in E_{h}$, $x_h,y_h \neq *$ and
        \begin{itemize}
            \item $P(\delta(y_h)-\delta(x_h)) = 0$ or
            \item there exists a finite set of horizontal edges $\{x^k_h \to y^k_h\}_k \sbs E_{h}$ and convex coefficients\footnote{{\it Convex coefficients} are a finite set of nonnegative real numbers that sum to 1.} $\{c_k\}_k \sbs [0,1]$ such that $\{x^k_h,y^k_h\}_k$ is pairwise disjoint and $P(\delta(y_h)-\delta(x_h)) = \frac12\sum_k c_k (\delta(y^k_h)-\delta(x^k_h))$.
        \end{itemize}
    \end{itemize}
\end{enumerate}
Let $r \in [0,\infty)$. Let us say that $P$ has {\it dyadic growth rate $\leq r$} if
\begin{enumerate}[resume]
    \item \fbox{Dyadic Growth Rate $\leq r$} there exists $C_3<\infty$ such that for all $x \in X$ and all $n \geq 0$, $|\{y \in X: x\in\supp P^n(\delta(y))\}| \leq C_3 2^{nr}$.
\end{enumerate}
\end{definition}

The two properties above are critical in proving our main results. They are used in tandem to prove uniform boundedness of representations we construct on $\ell^p$ for values of $p < \frac{r}{r-1}$ where $r$ is the dyadic growth rate (Lemma~\ref{lem:Lipschitz->bndd}). 

\begin{example}[Euclidean Dyadic Hyperbolic Fillings] \label{ex:Euclidean}
Recall the set of $n$th dyadic rationals $\Dyd_n$ from Definition~\ref{def:Euclidean}. For $n\geq 1$, we consider the Markov operator $P: \M_{00}(\Dyd_n) \to \M_{00}(\Dyd_{n-1})$ defined by
\begin{equation*}
    P(\delta(x)) = \begin{cases}
        \delta(x) & x \in \Dyd_{n-1} \\
        \frac12\delta(x-2^{-n})+\frac12\delta(x+2^{-n}) & x\in \Dyd_n\setminus \Dyd_{n-1}.
        \end{cases}
\end{equation*}

Let $m\in\N$. Recall the $m$-dimensional dyadic grid $\Dyd_n^m$ from Definition~\ref{def:Euclidean}. For $n\geq 1$, we form the product Markov operator $P_m: \M_{00}(\Dyd_n^m) \to \M_{00}(\Dyd_{n-1}^m)$ by
\begin{equation*}
    P_m(\delta(x_1,x_2,\dots x_m)) := P(\delta(x_1)) \otimes P(\delta(x_2)) \otimes \dots P(\delta(x_m)),
\end{equation*}
where the tensor product is defined on point mass measures by $\delta(y_1) \otimes \delta(y_2) \otimes \dots \delta(y_m) = \delta_{(y_1,y_2,\dots y_m)}$ and extended to all signed measures multilinearly.

Recall the Euclidean dyadic hyperbolic filling graph $\Dyd^m$ from Definition~\ref{def:Euclidean}. We define the Markov operator $P_m: \M_{00}(\Dyd^m) \to \M_{00}(\Dyd^m)$ by
\begin{equation*}
    P_m(\delta(n) \otimes \delta(u)) := \delta(n-1)\otimes P_m(\delta(u))
\end{equation*}
for all $\{n\}\times u \in \cup_{n\geq 1}\{n\}\times \Dyd^m_n$ and
\begin{equation*}
    P_m(\delta(y)) := \delta(*)
\end{equation*}
for all $y \in \cup_{n\in\{-1,0\}}\{n\}\times \Dyd^m_n$. Importantly, observe that $P_m(\M_{00}(\{n\}\times \Dyd_n^m)) \sbs \M_{00}(\{n-1\}\times \Dyd_{n-1}^m)$ for every $n\geq 0$. See Figure~\ref{fig:Markov-vertex} for an illustration of the action of $P_2$ and $P_2^2 = P_2 \circ P_2$ on a Dirac measure $\delta(u)$.
\end{example}

\begin{figure}
\centering

\begin{tikzpicture}[
    scale=.6,
    level0/.style={circle,draw=black,fill=red!80,inner sep=1.8pt},
    level1/.style={circle,draw=black,fill=blue!70,inner sep=1.5pt},
    level2/.style={circle,draw=black,fill=teal!60,inner sep=1.2pt},
    verticalEdge/.style={gray!75,thick},
    horizontalEdge/.style={black,thick},
    active/.style={circle,draw=red!80!black,fill=red!70,inner sep=2.5pt},
    prob/.style={
        font=\scriptsize,
        text=red!80!black,
        fill=white,
        inner sep=1pt,
        rounded corners=1pt,
        anchor=south east
    }
]

\def\Side{3.6}
\def\Slant{0.7}

\def\Yzero{5.5}
\def\Yone{0.8}
\def\Ytwo{-3.9}

\newcommand{\squaregrid}[5]{%
    \pgfmathsetmacro{\half}{\Side/2}
    \pgfmathsetmacro{\step}{\Side/(#2-1)}

    \foreach \r in {1,...,#2}{
        \foreach \c in {1,...,#2}{
            \pgfmathtruncatemacro{\k}{\c+(\r-1)*#2}
            \pgfmathsetmacro{\x}{#3-\half+(\c-1)*\step}
            \pgfmathsetmacro{\y}{#4+\half-(\r-1)*\step}
            \coordinate (c#1-\k) at (\x,\y);
        }
    }

    \foreach \r in {1,...,#2}{
        \pgfmathtruncatemacro{\a}{1+(\r-1)*#2}
        \pgfmathtruncatemacro{\b}{#2+(\r-1)*#2}
        \draw[horizontalEdge] (c#1-\a)--(c#1-\b);
    }

    \foreach \c in {1,...,#2}{
        \pgfmathtruncatemacro{\b}{\c+(#2-1)*#2}
        \draw[horizontalEdge] (c#1-\c)--(c#1-\b);
    }

    \pgfmathtruncatemacro{\N}{#2*#2}
    \foreach \k in {1,...,\N}{
        \node (#1-\k) at (c#1-\k) [#5] {};
    }
}

\newcommand{\refine}[4]{%
    \foreach \r in {1,...,#2}{
        \foreach \c in {1,...,#2}{
            \pgfmathtruncatemacro{\a}{\c+(\r-1)*#2}
            \pgfmathtruncatemacro{\b}{2*\c-1+(2*\r-2)*#4}
            \draw[verticalEdge] (#1-\a)--(#3-\b);
        }
    }
}

\newcommand{\massat}[2]{%
    \node[active] at (#1) {};
    \path (#1) ++(-0.18,0.18) node[prob] {$#2$};
}

\pgfmathsetmacro{\Xzero}{-1.2}
\pgfmathsetmacro{\Xone}{\Xzero+\Slant}
\pgfmathsetmacro{\Xtwo}{\Xzero+2*\Slant}
\pgfmathsetmacro{\half}{\Side/2}

\squaregrid{L0}{2}{\Xzero}{\Yzero}{level0}
\squaregrid{L1}{3}{\Xone}{\Yone}{level1}
\squaregrid{L2}{5}{\Xtwo}{\Ytwo}{level2}

\refine{L0}{2}{L1}{3}
\refine{L1}{3}{L2}{5}

\massat{L2-17}{1}

\massat{L1-4}{\frac14}
\massat{L1-5}{\frac14}
\massat{L1-7}{\frac14}
\massat{L1-8}{\frac14}

\massat{L0-1}{\frac{3}{16}}
\massat{L0-2}{\frac{1}{16}}
\massat{L0-3}{\frac{9}{16}}
\massat{L0-4}{\frac{3}{16}}

\node[anchor=east] at (\Xtwo-\half-1.0,\Ytwo-0.9) {$\delta(u)$};
\node[anchor=east] at (\Xone-\half-1.0,\Yone-0.6) {$P_2(\delta(u))$};
\node[anchor=east] at (\Xzero-\half-1.0,\Yzero) {$P_2^2(\delta(u))$};

\end{tikzpicture}

\caption{The Markov operator $P_2$ acting once $P_2(\delta(u))$ and twice $P_2^2(\delta(u))$ on a Dirac measure $\delta(u)$.}
\label{fig:Markov-vertex}
\end{figure}
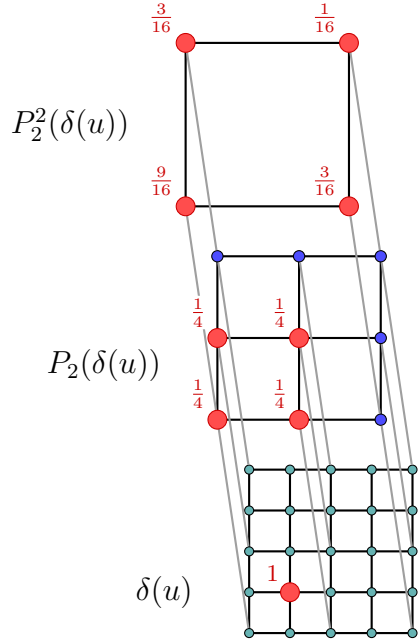

In the next theorem, we will see that the Markov operator of Example~\ref{ex:Euclidean} is admissible and Wasserstein $\frac12$-Contractive. One can imagine different ways of constructing admissible Markov operators on bounded degree hyperbolic graphs -- for example one could fix a basepoint and flow randomly (or even deterministically) along some geodesics towards that basepoint. The critical feature of the operator of Example~\ref{ex:Euclidean} is Wasserstein $\frac12$-Contractivity. This is a very delicate feature that is afforded to us by the precise combinatorial structure of the Euclidean dyadic hyperbolic filling.

\begin{theorem} \label{thm:admissible}
For every $m\in\N$, the Markov operator $P_m: \M_{00}(\Dyd^m) \to \M_{00}(\Dyd^m)$ of Example~\ref{ex:Euclidean} is admissible and Wasserstein $\frac12$-Contractive.
\end{theorem}

\begin{proof}
Let $m\in\N$.

\fbox{Gradation}. Gradation is clear from the construction: the grading is the natural one $\Dyd^m = \cup_{n\geq -1}\{n\}\times \Dyd_n^m$. The set $\{-1\}\times \Dyd_{-1}^m = \{*\}$ is an absorbing state by definition. It has already been observed that $P_m(\M_{00}(\{n\}\times \Dyd_n^m)) \sbs \M_{00}(\{n-1\}\times \Dyd_{n-1}^m)$ for every $n\geq 0$.

\fbox{Morse Property}. Let $n\geq 0$ and $x \in \Dyd_n$. Letting $0 \leq k \leq n$, we have by induction on $k$ that
\begin{align} \label{eq:Morse}
\nonumber \supp P^k(\delta(x)) &\sbs \left[x-\sum_{i=0}^{k-1} 2^{-n+i},x+\sum_{i=0}^{k-1} 2^{-n+i}\right] \cap \Dyd_{n-k} \\
    &= \left[x-2^{-n}(2^k-1),x+2^{-n}(2^k-1)\right] \cap \Dyd_{n-k}.
\end{align}
Now, the length of the interval $\left[x-2^{-n}(2^k-1),x+2^{-n}(2^k-1)\right]$ is strictly less than $2^{k+1-n}$, while the distance between consecutive points in $\Dyd_{n-k}$, when thought of as a metric subset of $[0,1]$, is $2^{k-n}$. Therefore, the intersection
$$\left[x-2^{-n}(2^k-1),x+2^{-n}(2^k-1)\right]~\cap~\Dyd_{n-k}$$
can contain at most 2 points, and hence the containment \eqref{eq:Morse} shows that
\begin{equation} \label{eq:|supp|<=2}
    |\supp P^k(\delta(x))| \leq 2.
\end{equation}

Now suppose that $\{n\}\times (x_1,x_2,\dots x_m) \in \cup_{n\geq 0} \{n\}\times \Dyd_n^m$. Then for $0 \leq k \leq n$, we have
\begin{align*}
    \supp &P_m^k(\delta(n)\otimes\delta(x_1,x_2,\dots x_m)) \\
    &= \{n-k\}\times\supp P^k(\delta(x_1)) \otimes P^k(\delta(x_2)) \otimes \dots P^k(\delta(x_m)) \\
    &\sbs \{n-k\}\times \supp P^k(\delta(x_1)) \times \supp P^k(\delta(x_2)) \times \dots \supp P^k(\delta(x_m)).
\end{align*}
By \eqref{eq:|supp|<=2}, this establishes the bound $|\supp P_m^k(\delta(n)\otimes\delta(x_1,x_2,\dots x_m))| \leq 2^m$, proving the Morse Property with $C_1 = 2^m$.

\fbox{Locality}. Let $\{n\}\times x = \{n\}\times (x_1,x_2,\dots x_m) \in \cup_{n\geq -1} \{n\}\times \Dyd_n^m$, and let $y \in \supp P_m(\delta(n)~\otimes~\delta(x))$. We will see that $d(\{n\}\times x, y) \leq m+1$.

First, the case $n=-1$ is trivial because $\{n\}\times x = y = *$. Second, the case $n=0$ is also trivial because $y = *$ and there is a vertical edge between $\{0\}\times x$ and $*$, showing that $d_{\Dyd^m}(\{n\}\times x, y) = 1$. Assume, then, that $n\geq 1$. Then $y$ is of the form $\{n-1\}\times (x_1+\eps_1, x_2+\eps_2,\dots x_m+\eps_m)$, where $\eps_i\in\{-2^{-n},0,2^{-n}\}$. Then
\begin{align*}
    d(\{n\}\times x, y) = |\{i\in\{1,2,\dots m\}: \eps_i \neq 0\}|+1 \leq m+1.
\end{align*}
This shows that Locality holds with $C_2=m+1$.

\fbox{Wasserstein $\frac{1}{2}$-Contractivity}.
The edge set $E$ of $\Dyd^m$ is by definition partitioned into vertical edges $E_v$ and horizontal edges $E_h$. Let $(x_v \to y_v)\in E_{v}$. Then $x_v = \{n\}\times(x_1,x_2,\dots x_m)$ for some $n\geq 0$ and $(x_1,x_2,\dots x_m)\in \Dyd_n^m$, and either $n=0$ and $y_v = *$, or $n \geq 1$ and $y_v = \{n-1\}\times(x_1,x_2,\dots x_m) \in \{n-1\}\times \Dyd_{n-1}^m$. In the first case, we have
\begin{equation*}
    P_m(\delta(x_v)) = \delta(*) = \delta(y_v),
\end{equation*}
and in the second case, we have
\begin{align*}
    P_m(\delta(x_v)) &= \delta(n-1)\otimes P(\delta(x_1))\otimes P(\delta(x_2))\otimes\dots P(\delta(x_m)) \\
    &= \delta(n-1)\otimes\delta(x_1)\otimes\delta(x_2)\otimes\dots\delta(x_m) \\
    &= \delta(y_v).
\end{align*}
This shows that the vertical edge set axiom is satisfied for $E_v$.

Now let $(x_h \to y_h) \in E_h$. Then $x_h = \{n\}\times x$ and $y_h = \{n\}\times y$ for some $n \geq 0$ and some edge $(x \to y) = ((x_1,x_2,\dots x_m) \to (y_1,y_2, \dots y_m))$ of $\Dyd_n^m$. Then, by definition of the product graph $\Dyd_n^m = \Dyd_n \square \Dyd_n \square \dots \Dyd_n$, there exists $i \in \{1,2,\dots m\}$ such that $x_j = y_j$ for all $j\in\{1,2,\dots m\}\setminus\{i\}$ and $y_i-x_i = 2^{-n}$. If $n = 0$, we have
\begin{align*}
    P_m(\delta(y_h) - \delta(x_h)) = \delta(*) -\delta(*) = 0.
\end{align*}
Otherwise, for $n\geq 1$, we may assume without loss of generality that $y_i \in \Dyd_{n-1}$ and $x_i \in \Dyd_{n}\setminus \Dyd_{n-1}$, and therefore 
\begin{align} \label{eq:horizontaledge}
\nonumber P(\delta(y_i))-P(\delta(x_i)) &= \delta(y_i) - \tfrac12\delta(x_i-2^{-n})-\tfrac12\delta(x_i+2^{-n}) \\
    &=\tfrac12\delta(x_i+2^{-n})-\tfrac12\delta(x_i-2^{-n}).
\end{align}
Then we have
\begin{align*}
    P_m(\delta&(y_h) - \delta(x_h)) \\
    =& \delta(n-1)\otimes P(\delta(y_1))\otimes \dots P(\delta(y_{i-1})) \otimes P(\delta(y_{i})) \otimes P(\delta(y_{i+1})) \otimes \dots P(\delta(y_{m})) \\
    &- \delta(n-1)\otimes P(\delta(x_1))\otimes \dots P(\delta(x_{i-1})) \otimes P(\delta(x_{i})) \otimes P(\delta(x_{i+1})) \otimes \dots P(\delta(x_{m})) \\
    \overset{\eqref{eq:horizontaledge}}{=}& \tfrac12\delta(n-1)\otimes P(\delta(x_1))\otimes \dots P(\delta(x_{i-1})) \otimes \delta(x_{i}+2^{-n}) \otimes P(\delta(x_{i+1})) \otimes \dots P(\delta(x_{m})) \\
    &- \tfrac12\delta(n-1)\otimes P(\delta(x_1))\otimes \dots P(\delta(x_{i-1})) \otimes \delta(x_{i}-2^{-n}) \otimes P(\delta(x_{i+1})) \otimes \dots P(\delta(x_{m})).
\end{align*}
Then, choosing finitely many vertices $\{(z_1^k,\dots z_{i-1}^k,z_{i+1}^k,\dots z_m^k)\}_k \sbs \Pi_{j\in\{1,\dots i-1,i+1,\dots m\}} \Dyd_{n-1}$ and convex coefficients $\{c_k\}_k\sbs[0,1]$ such that
\begin{equation*}
    P(\delta(x_1))\otimes \dots P(\delta(x_{i-1})) \otimes P(\delta(x_{i+1})) \otimes \dots P(\delta(x_{m})) = \sum_k c_k \delta(z_1^k,\dots z_{i-1}^k, z_{i+1}^k, \dots z_{m}^k)
\end{equation*}
(which we can find since $P$ is Markov), the two identities above show us that
\begin{align*}
    P_m(\delta(y_h) - \delta(x_h)) = \tfrac12\sum_k c_k(\delta(y_h^k)-\delta(x_h^k)),
\end{align*}
where
\begin{align*}
    x_h^k = \{n-1\}\times(z_1^k,\dots z_{i-1}^k,x_i-2^{-n}, z_{i+1}^k, \dots z_{m}^k), \\
    y_h^k = \{n-1\}\times(z_1^k,\dots z_{i-1}^k,x_i+2^{-n}, z_{i+1}^k, \dots z_{m}^k).
\end{align*}
since $x_h^k \to y_h^k$ is by definition an edge of $\{n-1\}\times \Dyd_{n-1}^m$, it is in $E_h$. Hence, the horizontal edge set properties are satisfied for $E_h$. This proves admissibility and Wasserstein $\frac12$-Contractivity of $P_m$. See Figure~\ref{fig:Markov-edge} for an illustration of a decomposition $P_m(\delta(y_h) - \delta(x_h)) = \tfrac12\sum_k c_k(\delta(y_h^k)-\delta(x_h^k))$ (denoted $\overline{P}(u\to v)$ in that figure) for a particular horizontal edge $x_h\to y_h = u \to v$ in the case $m=2$.
\end{proof}

\begin{definition}[Admissible Subgraphs]
Suppose $P$ is an admissible Markov operator on the connected, directed, bounded degree graph $(X,E)$ with basepoint $*$, grading $X = \sqcup_{n\geq-1} X_n$, and edge set partition $E = E_v \sqcup E_h$. Suppose that $Y\sbs X$ is an induced subgraph containing $*$ with its intrinsic shortest path metric $d_Y$. We say that $Y$ is an {\it admissible subgraph} of $X$, with respect to $P$, if
\begin{itemize}
    \item \fbox{Closure and Locality under $P$} there exists a constant $C_2'<\infty$ such that, whenever $x\in Y$ and $y \in \supp P(\delta(x))$, it holds that $y \in Y$ and $d_Y(x,y) \leq C_2'$.
\end{itemize}
In this case, it is easy to see that the restriction of $P$ to $Y$ is an admissible Markov operator on $Y$, and if $P$ is Wasserstein $\frac12$-Contractive then so is the restriction. Indeed, $P$ maps $\M_{00}(Y)$ to $\M_{00}(Y)$ by the item above, so the restriction is well-defined. The partition $Y = \sqcup_{n\geq -1} Y_n := \sqcup_{n\geq -1} Y \cap X_n$ satisfies the Gradation property, the Morse Property holds for $Y$ with the same constant $C_1$ that it does for $X$, Locality holds with constant $C_2'$ for $Y$ by the item above, and Wasserstein $\frac12$-Contractivity holds for the partition $(E_v \cap (Y\times Y)) \sqcup (E_h \cap (Y\times Y))$ of the edge set of $Y$.
\end{definition}

\begin{proposition} \label{prop:admissible}
Let $m\in\N$ and let $P_m$ be the admissible Markov operator on $\Dyd^m$ of Theorem~\ref{thm:admissible}. Let $Z \sbs [0,1]^m$ be compact and nonempty, and let $\Hyp(Z) \sbs \Dyd^m$ be the induced subgraph of Definition~\ref{def:subEuclidean}. Then $\Hyp(Z) \sbs \Dyd^m$ is an admissible subgraph of $\Dyd^m$ with respect to $P_m$. Moreover, if $Z$ has sharp Assouad dimension $\leq r$ for some $r\in[0,m]$, then the restriction of $P_m$ to $\Hyp(Z)$ has dyadic growth rate $\leq r$.
\end{proposition}

\begin{proof}
Let $x \in \Hyp(Z)$ and $y\in\supp P_m(\delta(x))$. If $y = *$, then $y \in \Hyp(Z)$ and $d_{\Hyp(Z)}(x,y) \leq 1$, verifying Closure and Locality under $P_m$ in this case. Assume, then, that $y \neq *$. Then there exists $n\geq 1$, $Q \in \D_{n}(Z)$, and $u \in Q^{(0)}$ such that $x = \{n\}\times u$. Since $y\in\supp P_m(\delta(x))$, it is easy to see that there exists $v \in \hat{Q}^{(0)}$ such that $y = \{n-1\}\times v$, where $\hat{Q} \in \D_{n-1}(Z)$ is the parent of $Q$. This shows that $y \in \Hyp(Z)$, verifying Closure under $P_m$. Furthermore, we obviously have that $\|u-v\| \leq 2^{-n+1}$ since the $\ell^\infty$-diameter of $\hat{Q}$ is $2^{-n+1}$, and thus Locality follows from Lemma~\ref{lem:Hyp(Z)distance}. This shows that $\Hyp(Z)$ is an admissible subgraph with respect to $P_m$.

To prove the last sentence, assume that $Z$ has sharp Assouad dimension $\leq r$ for some $r\in[0,m]$. Then by Lemma~\ref{lem:Assouad->dyadicgrowthrate}, the hyperbolic filling $\Hyp(Z)$ has dyadic growth rate $\leq r$ (in the sense of Definition~\ref{def:dyadicgrowthrate}). One can see from the definition of $P_m$ that if $x,y \in \Hyp(Z)$ and $n\geq 0$ with $x \in \supp P_m^n(\delta(y))$, then $y$ lies below $x$ (in the sense of Definition~\ref{def:dyadicgrowthrate}). From this, it is easy to see that $\Hyp(Z)$ having dyadic growth rate $\leq r$ (in the sense of Definition~\ref{def:dyadicgrowthrate}) implies that $P_m$ restricted to $\Hyp(Z)$ has dyadic growth rate $\leq r$ (in the sense of Definition~\ref{def:admissible}).
\end{proof}

\begin{definition}[Graph $K$-Leaf Extension]
\label{def:graphextension}
Suppose $P$ is an admissible Markov operator on the connected, directed, bounded degree graph $X$ with basepoint $*$ and grading $X = \sqcup_{n\geq-1} X_n$. Fix $K\in\N$, we form the {\it $K$-leaf extension} $\tilde{X}$ of $X$ by attaching $K$ leaves to each vertex of $X$ via a vertical edge.

Specifically, the graded vertex set is $\tilde{X} = \sqcup_{n\geq -1} \tilde{X}_n$ where $\tilde{X}_{-1} := X_{-1} = \{*\}$ and for $n\geq 0$,
$$\tilde{X}_n := X_n \sqcup \left(\{1,2,\dots K\} \times X_{n-1}\right).$$
Obviously, $X$ is naturally identified as a graded subset of $\tilde{X}$.

There are no new horizontal edges in $\tilde{X}$; all horizontal edges of $\tilde{X}$ are those that are already horizontal edges of $X$. The vertical edges of $X$ remain vertical edges of $\tilde{X}$, and we add to $\tilde{X}$ a new vertical edge $(i,x) \to x$ for every $n\geq 0$, $x \in X_{n-1}$, and $i\in\{1,2,\dots K\}$.

We define the Markov operator $\tilde{P}: \M_{00}(\tilde{X}) \to \M_{00}(\tilde{X})$ by
\begin{equation*}
    \tilde{P}(\delta(\tilde{x})) = \begin{cases}
        P(\delta(x)) & \tilde{x} = x \in X \\
        \delta(x) & \tilde{x} = (i,x) \in \tilde{X}\setminus X.
    \end{cases}
\end{equation*}
\end{definition}

The following proposition is immediate, and we leave the verification of details to the reader.

\begin{proposition} \label{prop:graphextension}
Suppose $P$ is an admissible Markov operator on the connected, directed, bounded degree graph $X$. Fix $K\in\N$, and let $\tilde{X}$ denote the $K$-leaf extension of $X$ and $\tilde{P}: \M_{00}(\tilde{X}) \to \M_{00}(\tilde{X})$ the linear operator of Definition~\ref{def:graphextension}. Then the following are true.
\begin{enumerate}
    \item $\tilde{X}$ is a connected, directed, bounded degree graph that is isometric to the metric $K$-leaf extension of $X$ of Definition~\ref{def:metricextension}.
    \item $\tilde{P}$ is an admissible Markov operator on $\tilde{X}$.
    \item If $P$ is Wasserstein $\frac12$-Contractive, then so is $\tilde{P}$.
    \item For any $r\in[0,\infty)$, if $P$ has dyadic growth rate $\leq r$, then $\tilde{P}$ has dyadic growth rate $\leq r$ (with constant increasing by a factor of at most $K$).
\end{enumerate}
\end{proposition}

\section{Actions on $\ell^p$}
\label{sec:lpactions}

Let $P$ be an admissible Markov operator and a connected, directed, bounded degree graph $X$. Let $\M_{*}(X) \sbs \M_{00}(X)$ denote the subspace consisting of all $\mu$ with $\mu(\{*\}) = 0$, and let $\M_X(X) \sbs \M_{00}(X)$ denote the subspace consisting of all $\mu$ with $\mu(X) = 0$. For $p\in [1,\infty]$, we can equip $\M_{*}(X)$ with the $\ell^p$-norm $\|\mu\|_p = \left(\sum_{x\in X} |\mu(\{x\})|^p \right)^{1/p}$, and we denote the resulting normed space by $\M_{*}^p(X)$.

Let $I: \M_{00}(X)\to\M_{00}(X)$ denote the identity. Define the {\it Laplacian} $\Delta = I - P: \M_{*}(X) \to \M_X(X)$ and the {\it inverse Laplacian} $\Delta^{-1} = \sum_{n\geq 0} P^n\one_{X\setminus\{*\}}: \M_X(X) \to \M_{*}(X)$ (whenever $\mu\in\M_{00}(X)$ and $A \sbs X$, we use the notation $\mu\one_A$ to denote the measure $(\mu\one_A)(B) := \mu(A\cap B)$). The Gradation property and finite support of $\mu \in \M_X(X)$ mean that $P^n(\mu)\one_{X\setminus\{*\}}$ equals 0 for $n$ sufficiently large (depending on $\mu)$, and hence the sum in the definition of $\Delta^{-1}$ is finite. Let us quickly check that $\Delta^{-1} \circ \Delta = id_{\M_{*}(X)}$ and $\Delta \circ \Delta^{-1} = id_{\M_X(X)}$, using the obvious facts that $\{\delta(x)-\delta(*)\}_{x\in X\setminus\{*\}}$ spans $\M_{X}(X)$ and $\{\delta(x)\}_{x\in X\setminus\{*\}}$ spans $\M_{*}(X)$. Let $x \in X\setminus\{*\} = \sqcup_{n\geq 0} X_n$, and let $N \geq 0$ with $x \in X_N$. Then we have
\begin{align*}
    \Delta\Delta^{-1}(\delta(x)-\delta(*)) &= (I-P)\sum_{n\geq 0}P^n(\delta(x)-\delta(*))\one_{X\setminus\{*\}} \\
    &= (I-P)\sum_{n=0}^N P^n(\delta(x)) \\
    &= \delta(x) - P^{N+1}(\delta(x)) \\
    &= \delta(x) - \delta(*),
\end{align*}
and
\begin{align*}
    \Delta^{-1}\Delta\delta(x) &= \sum_{n\geq 0}P^n((I-P)\delta(x))\one_{X\setminus\{*\}} \\
    &= \sum_{n=0}^N P^n(\delta(x)) - \sum_{n=1}^{N}P^n(\delta(x)) \\
    &= \delta(x),
\end{align*}
as claimed.

Let $f: X \to X$ be a map. It induces a {\it pushforward} operator $f_\#: \M_{00}(X) \to \M_{00}(X)$ that is linear and satisfies $f_\#(\delta(x)) = \delta(f(x))$. The operator $f_\#$ preserves the subspace $\M_X(X)$.

The following lemma is the important technical estimate powering the main results of this section.

\begin{lemma} \label{lem:Lipschitz->bndd}
Let $P$ be an admissible, Wasserstein $\frac12$-Contractive Markov operator, with dyadic growth rate $\leq r$ for some $r\in[0,\infty)$, on a connected, directed, bounded degree graph $X$. Then for any $p < \frac{r}{r-1}$ and any $L<\infty$, there exists $C'<\infty$ such that for any $L$-Lipschitz map $f: X \to X$, we have the operator norm bound $\|\Delta^{-1}f_\#\Delta\|_{p\to p} \leq C'$.
\end{lemma}

\begin{proof}
Let $p < \frac{r}{r-1}$ and $L<\infty$. Let $E$ denote the directed edge set of $X$ with partition $E = E_h \sqcup E_v$ witnessing Wasserstein $\frac{1}{2}$-Contractivity. We denote by $\ell_{00}^p(E)$ the free vector space over $E$ equipped with the $\ell^p$-norm $\|\sum_i c_i e_i\|_p := \left(\sum_i|c_i|^p\right)^{1/p}$. Define the boundary operator $\partial: \ell^p_{00}(E) \to \M_X(X)$ to be the unique linear map with $\partial(x\to y) = \delta(y)-\delta(x)$ whenever $(x\to y) \in E$ is a directed edge. Note that this map is bounded for every $p$. Let $f: X\to X$ be $L$-Lipschitz.

The strategy of the proof is as follows: we will construct auxiliary linear operators $\overline{\Delta},\overline{f},\overline{\Delta^{-1}}$ that fit into the following commutative diagram:

\begin{figure}[H]
\begin{tikzcd}[column sep=large, row sep=large]
&
\ell_{00}^p(E)
    \arrow[r, "\overline{f}"]
    \arrow[d, "\partial"]
&
    \ell_{00}^p(E_v) \oplus \ell_{00}^p(E_h)
    \arrow[r, "\overline{\Delta^{-1}}"]
    \arrow[d, "\partial"]
&
\M_*^p(X) \oplus \ell_{00}^p(E_h)
    \arrow[d, "id_{\M_*(X)} \oplus \partial"]
\\
\M_*^p(X)
    \arrow[ur, bend left=20, "\overline{\Delta}"]
    \arrow[r, "\Delta"]
&
\M_X(X)
    \arrow[r, "f_\#"]
&
\M_X(X)
    \arrow[r, "\Delta^{-1}"]
&
\M_*^p(X)
\end{tikzcd}
\caption{A commutative diagram involving auxiliary operators $\overline{\Delta},\overline{f},\overline{\Delta^{-1}}$ that are used to rewrite $\Delta^{-1}f_\#\Delta$ as $(id_{\M_*(X)} \oplus \partial)\overline{\Delta^{-1}} \, \overline{f} \, \overline{\Delta}$.}
\label{fig:cd}
\end{figure}

In particular, this diagram says that $\Delta^{-1}f_\#\Delta = (id_{\M_*(X)} \oplus \partial)\overline{\Delta^{-1}} \, \overline{f} \, \overline{\Delta}$. Then, to prove that $\|\Delta^{-1}f_\#\Delta\|_{p\to p} \leq C$ for some $C<\infty$ depending only on $L$, it suffices to prove that the $p\to p$ norms of the operators $\overline{\Delta},\overline{f},\overline{\Delta^{-1}}$ are finite, with $\|\overline{f}\|_{p\to p}$ depending only on $L$. We undertake these constructions in the remainder of the proof. \\

\noindent {\bf Construction of $\boldsymbol{\overline{\Delta}}$.}
By Locality of $P$ and the assumption that $X$ has bounded degree, it is simple to construct a bounded linear operator $\overline{\Delta}: \M_{*}^p(X) \to \ell^p_{00}(E)$ satisfying the commutativity $\Delta = \partial \overline{\Delta}$ of Figure~\ref{fig:cd}. Indeed, let us define $\overline{\Delta}$ on the basis elements $\{\delta(u)\}_{u \in X\setminus\{*\}} \sbs \M_*(X)$ and then extend by linearity. Let $u \in X\setminus\{*\}$. Since $P$ is Markov and satisfies the Locality property, there exist $N\in\N$ (independent of $u$) and a convex combination $\sum_i c_i\delta(u_i) \in \M_*(X)$ such that $P(\delta(u)) = \sum_i c_i\delta(u_i)$ and $\max_i d_X(u,u_i) \leq N$. This last condition implies that for each $i$, we can find a path of length $N$ from $u$ to $u_i$, i.e., a set of edges $\{e_i^j\}_{j=1}^N \sbs E$ such that $\partial \sum_{j=1}^N e_i^j = \delta(u)-\delta(u_i)$. Then we define
$$\overline{\Delta}(\delta(u)) := \sum_i c_i \sum_{j=1}^N e_i^j \in \ell_{00}^p(E).$$
We have by construction that $\Delta = \partial\overline{\Delta}$. Boundedness of $\overline{\Delta}$ follows from the fact that $X$ has bounded degree. \\

\noindent {\bf Construction of $\boldsymbol{\overline{f}}$.}
By Locality of $P$ and the assumption that $X$ has bounded degree, there is a constant $C<\infty$ such that for any $L$-Lipschitz $f: X \to X$, we can find a $C$-bounded linear operator $\overline{f}: \ell^p_{00}(E) \to \ell^p_{00}(E)$ satisfying the commutativity $f_\# \partial = \partial \overline{f}$ of Figure~\ref{fig:cd}. The construction is very similar to that of $\overline{\Delta}$ and we omit the details. \\

\noindent {\bf Construction of $\boldsymbol{\overline{\Delta^{-1}}}$.}
Now we want to find a bounded operator $\overline{\Delta^{-1}}: \ell^p_{00}(E) = \ell^p_{00}(E_v) \oplus \ell^p_{00}(E_h) \to \M_{*}^p(X) \oplus \ell^p_{00}(E_h)$ such that the commutativity
\begin{equation} \label{eq:Delta_inv_comm}
    \Delta^{-1} \partial = (id_{\M_{*}(X)}\oplus\partial) \overline{\Delta^{-1}}
\end{equation}
of Figure~\ref{fig:cd} is satisfied. Towards this end, we will find a bounded operator $\overline{P}: \ell^p_{00}(E_h)\to\ell^p_{00}(E_h)$ such that
\begin{equation} \label{eq:Pbar_comm}
    P\partial = \partial\overline{P}
\end{equation}
and
\begin{equation} \label{eq:contraction}
    \text{for some } 0\leq n_0 < \infty,\text{ the power }\overline{P}^{n_0} \text{ is a strict } p\to p \text{ contraction.}
\end{equation}
Once such a bounded operator $\overline{P}$ has been found, \eqref{eq:Pbar_comm} and \eqref{eq:Pbar_comm} will imply that the operator $\overline{\Delta^{-1}}:\ell^p_{00}(E_v) \oplus \ell^p_{00}(E_h) \to \M_{*}^p(X) \oplus \ell^p_{00}(E_h)$ defined by
\begin{equation*}
    \overline{\Delta^{-1}}(x_v \to y_v \, , \, x_h \to y_h) := (-\delta(x_v), \sum_{n\geq 0} \overline{P}^n(x_h \to y_h))
\end{equation*}
is a bounded linear operator satisfying the required commutativity \eqref{eq:Delta_inv_comm}. Indeed, the boundedness of $\overline{\Delta^{-1}}$ is clear from \eqref{eq:contraction}, and then given that \eqref{eq:Pbar_comm} holds, we have for any $(x_v \to y_v \, , \, x_h \to y_h) \in \ell^p_{00}(E_v) \oplus \ell^p_{00}(E_h) = \ell^p_{00}(E)$ that
\begin{align*}
    (id_{\M_{*}(X)}\oplus\partial) &\overline{\Delta^{-1}}(x_v \to y_v \, , \, x_h \to y_h) = (id_{\M_{*}(X)}\oplus\partial)(-\delta(x_v), \sum_{n\geq 0} \overline{P}^n(x_h \to y_h)) \\
    &= -\delta(x_v) + \partial\sum_{n\geq 0} \overline{P}^n(x_h \to y_h) \\
    &\overset{\eqref{eq:Pbar_comm}}{=} -\delta(x_v) + \sum_{n\geq 0} P^n\partial(x_h \to y_h) \\
    &= \left(\sum_{n\geq 0}P^n(P(\delta(x_v))-\delta(x_v)) + \sum_{n\geq 0} P^n\partial(x_h \to y_h)\right)\one_{X\setminus\{*\}} \\
    &= \left(\sum_{n\geq 0}P^n\partial(x_v \to y_v) + \sum_{n\geq 0} P^n\partial(x_h \to y_h)\right)\one_{X\setminus\{*\}} \\
    &= \Delta^{-1}\partial(x_v \to y_v,x_h \to y_h),
\end{align*}
verifying \eqref{eq:Delta_inv_comm}. We define the operator $\overline{P}$ in the next paragraph and verify that it is bounded and that \eqref{eq:Pbar_comm} and \eqref{eq:contraction}. We remark that the verification of \eqref{eq:contraction} is the most important technical estimate in the paper, and the only place where the conformal dimension $Q$ of $\partial\Gamma$ (eventually) enters the calculation of the permitted exponents $p$.

It is straightforward to define the operator $\overline{P}: \ell^p_{00}(E_h)\to\ell^p_{00}(E_h)$; by Wasserstein $\frac{1}{2}$-Contractivity, we can choose for any horizontal $(x_h \to y_h) \in E_h$ a finite set of horizontal edges $\{x^i_h \to y^i_h\}_i \sbs E_{h}$ and convex coefficients $\{c_i\}_i \sbs [0,1]$ such that $\{x^i_h,y^i_h\}_i$ is pairwise disjoint and $P(\delta(y_h)-\delta(x_h)) = 0$ or $P(\delta(y_h)-\delta(x_h)) = \frac12\sum_i c_i (\delta(y^i_h)-\delta(x^i_h))$. Then we define $\overline{P}(x_h \to y_h) := 0$ if $P(\delta(y_h)-\delta(x_h)) = 0$ and $\overline{P}(x_h \to y_h) := \frac12\sum_i c_i(x^i_h \to y^i_h)$ if $P(\delta(y_h)-\delta(x_h)) = \frac12\sum_i c_i (\delta(y^i_h)-\delta(x^i_h))$. See Figure~\ref{fig:Markov-edge} for an illustration of $\overline{P}$ when $X = \Dyd^2$.

\begin{figure}
\centering
\begin{tikzpicture}[
    scale=.6,
    level0/.style={circle,draw=black,fill=red!80,inner sep=1.8pt},
    level1/.style={circle,draw=black,fill=blue!70,inner sep=1.5pt},
    level2/.style={circle,draw=black,fill=teal!60,inner sep=1.2pt},
    verticalEdge/.style={gray!75,thick},
    horizontalEdge/.style={black,thick},
    activeEdge/.style={
        ->,
        >=stealth,
        red!80!black,
        line width=1.8pt
    },
    prob/.style={
        font=\scriptsize,
        text=red!80!black,
        fill=white,
        inner sep=1pt,
        rounded corners=1pt
    }
]

\def\Side{3.6}
\def\Slant{0.7}

\def\Yzero{5.5}
\def\Yone{0.8}
\def\Ytwo{-3.9}

\newcommand{\squaregrid}[5]{%
    \pgfmathsetmacro{\half}{\Side/2}
    \pgfmathsetmacro{\step}{\Side/(#2-1)}

    \foreach \r in {1,...,#2}{
        \foreach \c in {1,...,#2}{
            \pgfmathtruncatemacro{\k}{\c+(\r-1)*#2}
            \pgfmathsetmacro{\x}{#3-\half+(\c-1)*\step}
            \pgfmathsetmacro{\y}{#4+\half-(\r-1)*\step}
            \coordinate (c#1-\k) at (\x,\y);
        }
    }

    \foreach \r in {1,...,#2}{
        \pgfmathtruncatemacro{\a}{1+(\r-1)*#2}
        \pgfmathtruncatemacro{\b}{#2+(\r-1)*#2}
        \draw[horizontalEdge] (c#1-\a)--(c#1-\b);
    }

    \foreach \c in {1,...,#2}{
        \pgfmathtruncatemacro{\b}{\c+(#2-1)*#2}
        \draw[horizontalEdge] (c#1-\c)--(c#1-\b);
    }

    \pgfmathtruncatemacro{\N}{#2*#2}
    \foreach \k in {1,...,\N}{
        \node (#1-\k) at (c#1-\k) [#5] {};
    }
}

\newcommand{\refine}[4]{%
    \foreach \r in {1,...,#2}{
        \foreach \c in {1,...,#2}{
            \pgfmathtruncatemacro{\a}{\c+(\r-1)*#2}
            \pgfmathtruncatemacro{\b}{2*\c-1+(2*\r-2)*#4}
            \draw[verticalEdge] (#1-\a)--(#3-\b);
        }
    }
}

\newcommand{\edgemass}[3]{%
    \draw[activeEdge] (#1) -- (#2)
        node[midway,above=3pt,prob] {$#3$};
}

\pgfmathsetmacro{\Xzero}{-1.2}
\pgfmathsetmacro{\Xone}{\Xzero+\Slant}
\pgfmathsetmacro{\Xtwo}{\Xzero+2*\Slant}
\pgfmathsetmacro{\half}{\Side/2}

\squaregrid{L0}{2}{\Xzero}{\Yzero}{level0}
\squaregrid{L1}{3}{\Xone}{\Yone}{level1}
\squaregrid{L2}{5}{\Xtwo}{\Ytwo}{level2}

\refine{L0}{2}{L1}{3}
\refine{L1}{3}{L2}{5}

\edgemass{L2-17}{L2-18}{1}

\edgemass{L1-4}{L1-5}{\frac14}
\edgemass{L1-7}{L1-8}{\frac14}

\edgemass{L0-1}{L0-2}{\frac{1}{16}}
\edgemass{L0-3}{L0-4}{\frac{3}{16}}

\node[anchor=east]
    at (\Xtwo-\half-1.0,\Ytwo-0.9)
    {$u\to v$};

\node[anchor=east]
    at (\Xone-\half-1.0,\Yone-0.6)
    {$\overline{P}(u\to v)$};

\node[anchor=east]
    at (\Xzero-\half-1.0,\Yzero)
    {$\overline{P}^{\,2}(u\to v)$};

\end{tikzpicture}

\caption{The action of $\overline{P}$ and $\overline{P}^2$ on a horizontal edge $u\to v$ in the Euclidean dyadic hyperbolic filling of $[0,1]^2$.}
\label{fig:Markov-edge}
\end{figure}

The commutativity \eqref{eq:Pbar_comm} easily holds by construction. We also obviously have that the operator norm of $\overline{P}$ is $\frac{1}{2}$ when $p=1$. We will then compute the operator norm for $p=\infty$ and use interpolation to get an estimate for all $p\in[1,\infty]$. The bound for the $\infty\to\infty$ operator norm uses the dyadic growth rate $\leq r$ assumption, and it is the only place in this proof where this assumption is used.

It follows from the definition of dyadic growth rate $\leq r$ that there is a constant $C_{grw}<
\infty$ such that for all $n\geq 0$ and all $(x'_h \to y'_h) \in E_h$, there are at most $C_{grw}2^{nr}$ edges $(x_h \to y_h)\in E_h$ for which the $(x'_h \to y'_h)$-coefficient of $\overline{P}^n(x_h \to y_h)$ is nonzero. It follows from this that the operator norm of $\overline{P}^n$ when $p=\infty$ is $\leq 2^{-n}C_{grw}2^{nr} = C_{grw}2^{n(r-1)}$. Thus, by (Riesz-Thorin) interpolation, the operator norm of $\overline{P}^n$ for $p\in[1,\infty]$ is at most
\begin{align} \label{eq:Pbarbound}
    \|\overline{P}^n\|_{p\to p} \leq (2^{-n})^{\frac1p}(C_{grw}2^{n(r-1)})^{1-\frac1p} = C_{grw}^{1-\frac1p}2^{-n}2^{rn(1-\frac1p)}.
\end{align}
Recalling that $p < \frac{r}{r-1}$, we get that $1-\frac1p < \frac1r$. Then, setting $\eps = \frac1r - (1-\frac1p)>0$, we make this substitution into the bound \eqref{eq:Pbarbound} and obtain
\begin{align*}
    \|\overline{P}^n\|_{p\to p} \leq C_{grw}^{1-\frac1p}2^{-n}2^{rn(\frac1r-\eps)} = C_{grw}^{1-\frac1p}2^{-rn\eps}.
\end{align*}
Hence, $\|\overline{P}^n\|_{p\to p}<\infty$ for every $n\geq 0$, and if $n_0 > \frac{\log_2(C_{grw}^{1-\frac1p})}{r\eps}$, then $\|\overline{P}^{n_0}\|_{p\to p}<1$. This verifies \eqref{eq:contraction} and completes the proof.
\end{proof}

The following lemma and theorem are consequences of Lemma~\ref{lem:Lipschitz->bndd}.

\begin{lemma} \label{lem:complemented}
Let $P$ be an admissible, Wasserstein $\frac12$-Contractive Markov operator, with dyadic growth rate $\leq r$ for some $r\in[0,\infty)$, on a connected, directed, bounded degree graph $X$. Let $R \sbs X$ be a Lipschitz retract. Then for any $p < \frac{r}{r-1}$, the subspace $\Delta^{-1}\M_{X}(R) \sbs \Delta^{-1}\M_{X}(X) = \M^p_{*}(X)$ is the image of a bounded linear projection, where $\M_X(R) \sbs \M_X(X)$ denotes the subspace of all $\mu\in\M_X(X)$ with $\mu(X\setminus R) = 0$.
\end{lemma}

\begin{proof}
Let $p < \frac{r}{r-1}$. Let $\pi: X \to R$ be a Lipschitz retraction. We then consider the operator $\Delta^{-1}\pi_\#\Delta: \M^p_{*}(X) \to \Delta^{-1}\M_{X}(R)$. This operator is $p\to p$ bounded by Lemma~\ref{lem:Lipschitz->bndd}, and it is easy to see that it is a projection since $\pi$ is a retract.
\end{proof}

Recall that a uniformly Lipschitz action of a finitely generated group $\Gamma$ on a metric space $(X,d)$ has {\it compression exponent $\beta$} if for some point $x_0\in X$, there exists $C<\infty$ such that $d(\gamma\cdot x_0,x_0) \geq C^{-1}|\gamma|^\beta-C$ for all $\gamma\in\Gamma$, where $|\cdot|$ denote the word length.

\begin{theorem} \label{thm:Markov->lp}
Let $P$ be an admissible, Wasserstein $\frac12$-Contractive Markov operator, with dyadic growth rate $\leq r$ for some $r\in[0,\infty)$, on a connected, directed, bounded degree graph $X$. Let $R \sbs X$ be a Lipschitz retract. Let $\Gamma$ be a group acting on $R$ by uniformly Lipschitz maps with biLipschitz embedded orbit. Then for all $p<\frac{r}{r-1}$, $\Gamma$ acts properly on $\ell^p$ by uniformly Lipschitz affine maps with compression exponent $1/p$.
\end{theorem}

\begin{proof}
Let $p < \frac{r}{r-1}$. It suffices to produce the action on a subspace of $\M_{*}^p(X)$ that is the image of a bounded linear projection, since this induces an action on $\ell^p$ in this obvious way. Let $\pi:X \to R$ be a Lipschitz retract. Replacing $R$ by $R' := (R\setminus\{\pi(*)\})\cup\{*\}$, where $*$ is the basepoint of $X$, we may assume that $* \in R$. This is because the transposition $\tau: X \to X$ that swaps $*$ and $\pi(*)$ is a biLipschitz equivalence between $R$ and $R'$, and hence $\tau$ conjugates the uniformly Lipschitz action of $\Gamma$ on $R$ to a uniformly Lipschitz action of $\Gamma$ on $R'$, and $\tau \circ \pi: X \to R'$ is a Lipschitz retraction.

Since, by Lemma~\ref{lem:complemented}, the subspace $\Delta^{-1}\M_{X}(R) \sbs \M^p_{*}(X)$ is the image of a bounded linear projection, it suffices to produce the required action on this subspace. If $\gamma: R \to R$ denotes the Lipschitz map associated to $\gamma\in\Gamma$, then the linear action of $\gamma$ on $\Delta^{-1}\M_{X}(R)$ is given by $\Delta^{-1}\gamma_\#\Delta$. This clearly satisfies the axioms of a group action, and $\sup_{\gamma\in\Gamma}\|\Delta^{-1}\gamma_\#\Delta\|_{p\to p}<\infty$ by Lemma~\ref{lem:Lipschitz->bndd}. It remains to produce a cocycle with compression exponent $1/p$.

We define $b: \Gamma \to \Delta^{-1}\M_{X}(R)$ by $b(\gamma) = \Delta^{-1}(\delta(\gamma\cdot *)-\delta(*))$. This is clearly a cocycle of the representation defined above. Let $\gamma\in\Gamma$. Set $D := d(\gamma\cdot *,*)$. We will show that $\|b(\gamma)\|_p^p \geq D/C' - C'$ for some $C'<\infty$, demonstrating a compression exponent of $1/p$. The Gradation, Morse, and Locality properties imply that there exists $C<\infty$ such that
\begin{itemize}
    \item the measures $\{P^n(\delta(\gamma\cdot*)\}_{n\geq 0}$ have pairwise disjoint support,
    \item for all $n \geq 0$, $P^n(\delta(\gamma\cdot*))$ is a probability measure supported on at most $C$ points, and
    \item for all $n < D/C$, we have $P^n(\delta(\gamma\cdot*)) \sbs \M_{00}(R\setminus\{*\})$.
\end{itemize}
Using these facts, we get
\begin{align*}
    \|b(\gamma)\|_p^p &= \|\Delta^{-1}(\delta(\gamma\cdot*)-\delta(*))\|_p^p \\
    &= \|\sum_{n\geq 0} P^n\one_{X\setminus\{*\}}(\delta(\gamma\cdot*)-\delta(*))\|_p^p \\
    &\geq \sum_{n=0}^{\lfloor D/C \rfloor} \|P^n(\delta(\gamma\cdot*))\|_p^p \\
    &\geq \frac{D/C-1}{C^{p}}.
\end{align*}
\end{proof}

\begin{theorem} \label{thm:hyperbolic->Markov}
Let $\Gamma$ be a finitely generated hyperbolic group. Let $Q$ denote the Assouad conformal dimension of $\partial\Gamma$. Then for all $r>Q$, there exists an admissible, Wasserstein $\frac12$-Contractive Markov operator, with dyadic growth rate $\leq r$, on a connected, directed, bounded degree graph $\tilde{X}$ such that $\Gamma$ is biLipschitz equivalent to a Lipschitz retract of $\tilde{X}$.
\end{theorem}

\begin{proof}
Let $r > Q$ and set $\eps := r-Q >0$. Let $\tilde{Z}$ be a metric space with sharp Assouad dimension $\leq Q+\eps/2$ such that $\partial\Gamma$ is power quasisymmetrically equivalent to $\tilde{Z}$. Replace $\tilde{Z}$ by its $(\frac{Q+\eps/2}{Q+\eps})$-snowflake, so that $\tilde{Z}$ now has sharp Assouad dimension $\leq Q+\eps = r$. Since the snowflake exponent is $<1$, the Assouad embedding theorem (\cite[Theorem~12.2]{Heinonen}) implies that $\tilde{Z}$ is biLipschitz equivalent to some $Z \sbs [0,1]^m$, for some $m\in\N$. Of course, $Z$ has sharp Assouad dimension $\leq r$ since $\tilde{Z}$ does, and $\partial\Gamma$ is power quasisymmetrically equivalent to $Z$.

By Theorem~\ref{thm:qifilling}, $\Gamma$ is quasi-isometric to $X := \Hyp(Z)$. Thus, by Lemma~\ref{lem:qi->biLip2} and Proposition~\ref{prop:graphextension}(1), we can find $K \in \N$ such that $\Gamma$ is biLipschitz equivalent to a coarsely dense subset of the $K$-leaf extension $\tilde{X}$ of $X$. Note that any nearest-neighbor projection from a uniformly discrete metric space onto a coarsely dense subset is Lipschitz, and therefore $\Gamma$ is biLipschitz equivalent to a Lipschitz retract $R$ of $\tilde{X}$.

Now, by Proposition~\ref{prop:admissible}, there exists an admissible, Wasserstein $\frac12$-Contractive Markov operator, with dyadic growth rate $\leq r$, on $X$. Then by Proposition~\ref{prop:graphextension}(2)(3)(4), there exists an admissible, Wasserstein $\frac12$-Contractive Markov operator, with dyadic growth rate $\leq r$, on $\tilde{X}$. This completes the proof.
\end{proof}

We are now ready to prove our main result, Theorem~\ref{thm:main}, restated here for the convenience of the reader.

\begin{theorem} \label{thm:main2}
Let $\Gamma$ be a finitely generated hyperbolic group. Let $Q$ denote the Assouad conformal dimension of $\partial\Gamma$. Then for all $p < \frac{Q}{Q-1}$, the group $\Gamma$ admits a proper uniformly Lipschitz affine action on $\ell^p$ with compression exponent $1/p$.
\end{theorem}

\begin{proof}
Let $p < \frac{Q}{Q-1}$. Then we can find $r>Q$ with $p < \frac{r}{r-1}$. By Theorem~\ref{thm:hyperbolic->Markov}, there exists an admissible, Wasserstein $\frac12$-Contractive Markov operator, with dyadic growth rate $\leq r$, on a connected, directed, bounded degree graph $\tilde{X}$ such that $\Gamma$ is biLipschitz equivalent to a Lipschitz retract $R$ of $\tilde{X}$. Since $\Gamma$ is biLipschitz equivalent to $R$ and acts isometrically on itself with isometrically embedded orbit, it acts by uniformly Lipschitz maps on $R$ with biLipschitzly embedded orbit. Then by Theorem~\ref{thm:Markov->lp}, $\Gamma$ admits a proper uniformly Lipschitz affine action on $\ell^p$ with compression exponent $1/p$.
\end{proof}

\subsection*{AI Usage Statement} AI tools (Gemini 3.5, Claude Opus 5, and ChatGPT 5.6) were used to search for references and to create Figures~\ref{fig:Hyp(Z)}, \ref{fig:Markov-vertex}, \ref{fig:cd}, and \ref{fig:Markov-edge}. All mathematical ideas in the paper were developed by the authors by 2025, and no AI tools were used in that process.

\bibliographystyle{alpha}

\newpage

\appendix
\section{} \label{app}

In this Appendix we prove Theorem~\ref{thm:qifilling}. We will rely on Bonk-Schramm's work \cite{BS} establishing an equivalence between the category of visual hyperbolic metric spaces with quasi-isometries, and the category of bounded, complete metric spaces with power quasisymmetries. We will also directly use the work of Bj{\"o}rn-Bj{\"o}rn-Shanmugalingam \cite{BBS} who constructed a hyperbolic filling similar to ours and established its basic properties.  

\subsection{Bonk-Schramm's equivalence of categories}

Let $(X,d_X)$ be a metric space, $K<\infty$, and $I = [0,b]$ or $I = [0,\infty)$. The image of $I$ under a map $f: I \to X$ satisfying $|s-t| - K \leq d_X(f(s),f(t)) \leq |s-t| + K$ is called a \emph{$K$-rough geodesic} if $I = [0,b]$ and a \emph{$K$-rough geodesic ray} if $I = [0,\infty)$. If $I = [0,b]$, the points $f(0),f(b)$ are called the \emph{endpoints} of the rough geodesic, and if $I = [0,\infty)$, the rough geodesic ray is said to {\it emanate} from the point $f(0)$. A metric space is \emph{roughly geodesic} if there exists $K<\infty$ such that every pair of points are the endpoints of some $K$-rough geodesic. A metric space is \emph{visual} if there exist $o\in X$ and $K' <\infty$ such that $X$ is the union of all $K'$-rough geodesic rays emanating from $o$. It follows from the quasi-isometric stability of geodesics \cite[Theorem~1.3.2]{Buyalo} that visuality is a quasi-isometric invariant of geodesic hyperbolic metric spaces. Finitely generated hyperbolic groups are known to be visual (see, for example, \cite[Lemma~3.1]{Bestvina}), and visual hyperbolic spaces are known to be roughly geodesic \cite[Proposition~5.6]{BS}.

We can now recall the equivalence of categories of Bonk-Schramm.

\begin{theorem}[\cite{BS}] \label{thm:BS}
Let $X$ and $X'$ be visual, hyperbolic metric spaces. Then $X$ and $X'$ are quasi-isometrically equivalent if and only if $\partial X$ and $\partial X'$ are power quasisymmetrically equivalent.
\end{theorem}

\begin{proof}
The ``only if" direction follows from \cite[Theorem~6.5(2)]{BS} and the fact that visual hyperbolic metric spaces are roughly geodesic (\cite[Proposition~5.6]{BS}). The ``if" follows from \cite[Theorem~7.4]{BS} and \cite[Theorem~8.2]{BS}.
\end{proof}

\subsection{The Hyperbolic filling of Bj{\"o}rn-Bj{\"o}rn-Shanmugalingam}
\label{ss:BBS}
Let $(Z,d)$ be a nonempty metric space with $\diam(Z)<1$. Let $\{A_n\}_{n\geq 0}$ be an increasing sequence of subsets of $Z$ such that $A_n$ is a maximal $2^{-n}$-separated subset for each $n \geq 0$, by which we mean $d(x,y) \geq 2^{-n}$ for every $x\neq y \in A_n$, and $A_n$ is not strictly included in any other subset with this property. Such sequences always exist by Zorn's Lemma. Note that the assumption $\diam(Z)<1$ guarantees $|A_0|=1$.

We will form a graph from $\{A_n\}_{n\geq 0}$. The vertex set is
\begin{equation*}
\Hyp_{\rm BBS}(Z) := \cup_{n=0}^{\infty} V_n, \quad \text{where } V_n = \{(n,x) : x \in A_n\}.
\end{equation*}
Two distinct vertices $(n,x), (m,y) \in \Hyp_{\rm BBS}(Z)$ will form an edge $(n,x) \sim (m,y)$ if $|n - m| \le 1$ and
\begin{align*}
B_{2^{-n+1}}(x) \cap B_{2^{-n+1}}(y) &\neq \emptyset, \quad \text{if } m = n, \\
B_{2^{-n}}(x) \cap B_{2^{-n}}(y) &\neq \emptyset, \quad \text{if } m = n \pm 1.
\end{align*}
Edges of the first type will be called {\it horizontal} and the second type {\it vertical}. We define the {\it hyperbolic filling (in the sense of \cite{BBS})} to be the graph $\Hyp_{\rm BBS}(Z)$ formed by the vertex set $V$ together with the edge set defined above. We equip $\Hyp_{\rm BBS}(Z)$ with its (unweighted) shortest path metric $d_{\Hyp_{\rm BBS}(Z)}$ (the condition $|A_0|=1$ ensures  that the graph is connected).

\begin{remark}
By including each edge of the graph in the space $d_{\Hyp_{\rm BBS}(Z)}$ as an isometric image of the unit interval, we may assume that $d_{\Hyp_{\rm BBS}(Z)}$ is a geodesic metric space, although we will sometimes shift our point-of-view and think of $d_{\Hyp_{\rm BBS}(Z)}$ as a metric only on the vertex set. Since we are concerned with the space $d_{\Hyp_{\rm BBS}(Z)}$ only up to quasi-isometry, this shift in viewpoint creates no problem. We treat the space $\Hyp(Z)$ of Definition~\ref{def:subEuclidean} in the same way.
\end{remark}

We recall the following facts from \cite{BBS}.

\begin{proposition}[\cite{BBS}] \label{prop:BBS}
For any nonempty, compact, doubling metric space $(Z,d)$ with $\diam(Z) < 1$, the filling $\Hyp_{\rm BBS}(Z)$ is proper, visual, and hyperbolic, and $\partial \Hyp_{\rm BBS}(Z)$ is power quasisymmetric to $Z$.
\end{proposition}

Before proving the proposition, we need to recall the following fact. When $X$ is a proper, geodesic, hyperbolic metric space, there is an alternate description of $\partial X$ in terms of equivalence classes of geodesic rays emanating from a fixed basepoint $o$. Let $o\in X$. We say that two geodesic rays $\gamma_1,\gamma_2: [0,\infty) \to X$ with $\gamma_i(0) = o$ for $i\in\{1,2\}$ are {\it equivalent} if $\sup_t d(\gamma_1(t),\gamma_2(t)) < \infty$. Let us denote the set of equivalence classes by $\partial_g X$. By \cite[Proof of Lemma~3.13]{BH}, the map $\partial_g X \to \partial _X$ defined by $[\gamma] \mapsto \{\gamma(i)\}_{i=1}^\infty$ is a bijection.

\begin{proof}[Proof of Proposition~\ref{prop:BBS}]
That $\Hyp_{\rm BBS}(Z)$ is visual is \cite[Corollary~3.2(d)]{BBS}, that it is hyperbolic is \cite[Theorem~3.4]{BBS}, and that it is proper follows from \cite[Proposition~4.5]{BBS}. The identification of the boundary of $\Hyp_{\rm BBS}(Z)$ with $Z$ is \cite[Proposition~4.4]{BBS}, with the caveat that the definition of boundary they work with is not defined in exactly the same way as in \S\ref{ss:hyperbolic}. Instead of attempting to explain the connection between the two notions of boundary, we will simply use intermediate results of \cite{BBS} to prove in the next paragraph that $\partial \Hyp_{\rm BBS}(Z)$ is snowflake equivalent to $Z$.

By \cite[Corollary~3.2(b)]{BBS}, every geodesic ray $\gamma:[0,\infty) \to \Hyp_{\rm BBS}(Z)$ emanating from the basepoint of $\Hyp_{\rm BBS}(Z)$ consists solely of vertical edges. By definition, this means that $\gamma(n) = (n,a_n)_{n\geq 0}$ where $a_n \in A_n$ and $d(a_{n+1},a_n) \leq 2^{-n}$ for all $n\geq 0$. The sequence $(a_n)_n$ is obviously Cauchy, and so we map such a ray to the limit in $Z$ of this sequence. The estimate in \cite[Lemma~3.3]{BBS} says that this map is a snowflake embedding (and hence a power quasi-symmetric embedding) from $\partial\Hyp_{\rm BBS}(Z)$ into $Z$. Finally, by maximality of each $A_n$, the set $\cup_{n\geq 0} A_n$ is obviously dense in $Z$ and also is contained in the image of the map, so the map is surjective since it's continuous (being a snowflake embedding) and has compact domain \cite[Exercise~2.3.6]{Buyalo}.
\end{proof}

We will need to observe one additional fact about the structure of $\Hyp_{\rm BBS}(Z)$ which will allow us to prove, in Proposition~\ref{prop:EuclideanqiBBS}, that $\Hyp_{\rm BBS}(Z)$ is quasi-isometrically equivalent to $\Hyp(Z)$.

\begin{lemma} \label{lem:BBS}
Let $(Z,d)$ be a nonempty and compact metric space with $\diam(Z) < 1$. Let $\Hyp_{\rm BBS}(Z) = \cup_{n\geq 0} V_n$ be the hyperbolic filling defined in this subsection. For every constant $C<\infty$, there exists another constant $C'<\infty$ such that for all $n\geq 0$, $\eps \in \{0,1\}$, $(n,u) \in V_n$, and $(n+\eps,v)\in V_{n+\eps}$, if $d(u,v) \leq C2^{-n}$, then $d_{\Hyp_{\rm BBS}(Z)}((n,u),(n+\eps,v)) \leq C'$.
\end{lemma}

\begin{proof}
Let $C<\infty$, $n\geq 0$, $\eps \in \{0,1\}$, $(n,u) \in V_n$, and $(n+\eps,v) \in V_{n+\eps}$ with $d(u,v) \leq C2^{-n}$. Let $\{A_n\}_{n\geq 0}$ be an increasing sequence of maximal $2^{-n}$-separated subsets of $Z$ used to construct $\{V_n\}_{n\geq 0}$ in the definition of $\Hyp_{\rm BBS}(Z)$. It follows from $d(u,v) \leq C2^{-n}$ that there exist a number $0\leq k \leq n$, depending only on $C$, and a point $a \in A_{n-k}$ such that $u,v \in B_{2^{k-n}}(a)$. We claim that $(n,u)$ can be connected to $(n-k,a)$ with a path of length at most $k+1$. Once this is established, it follows symmetrically that the same is true of $v$ (but with length $k+1+\eps$), which implies $d_{\Hyp_{\rm BBS}(Z)}((n,u),(n+\eps,v)) \leq 2k+1+\eps$, proving the lemma. It remains to prove the claim.

First suppose that $k=0$. Then $u,a \in A_n$ and $u\in B_{2^{-n}}(a)$, showing that there is a horizontal edge between $(n,u)$ and $(n,a)$ (unless they are equal, in which case the claim is trivially true). This proves the claim in the case $k=0$. Now assume that $k \geq 1$. By maximality of $A_j$, there exists a sequence of points $a_j \in A_j$ for $j \in \{n-k,n-k+1,\dots n\}$ such that $a_n = u$ and $d(a_{j-1},a_j) \leq 2^{-j+1}$ for all $n-k < j \leq n$. Note that this implies $(j,a_j)$ and $(j-1,a_{j-1})$ are joined by a vertical edge for all $n-k < j \leq n$. Summing these distances over $j$ and applying the triangle inequality yields $d(u,a_{n-k}) < 2^{k-n+1}$. Since we also have $d(u,a) \leq 2^{k-n}$, this shows that $(n-k,a)$ and $(n-k,a_{n-k})$ are joined by a horizontal edge (unless they are equal). Hence, we have constructed a path of length at most $k+1$ ($k$ vertical edges, at most 1 horizontal edge) joining $(n,u)$ to $(n-k,a)$.
\end{proof}

\begin{proposition} \label{prop:EuclideanqiBBS}
Let $m\in\N$ and $Z \sbs [0,1]^m$ be nonempty and compact with $\diam(Z)<1$. Then $\Hyp(Z)$ (see Definition~\ref{def:subEuclidean}) and $\Hyp_{\rm BBS}(Z)$ are quasi-isometrically equivalent.
\end{proposition}

\begin{proof}
Let $\{A_n\}_{n\geq 0}$ be an increasing sequence of maximal $2^{-n}$-separated subsets of $Z$ used to construct $V_n = \{n\}\times A_n$ in the definition of $\Hyp_{\rm BBS}(Z) = \cup_{n\geq 0} V_n$.

First we will define a map $f: \Hyp(Z) \to \Hyp_{\rm BBS}(Z)$. We leave the map undefined on the basepoint $*$ (the value of $f$ at a single point in a bounded degree graph is irrelevant for a quasi-isometry). Every other point in $\Hyp(Z)$ is of the form $\{n\}\times u$ where $n\geq 0$ and $u \in Q^{(0)}$ for some $Q \in \D_n(Z)$. Let $a_{n,u}$ be any point in $A_n$ for which $u \in B_{2^{-n}}(a_{n,u})$, which exists by maximality of $A_n$. Define
$$f(\{n\}\times u) := (n,a_{n,u}) \in \Hyp_{\rm BBS}(Z).$$
Let us prove that $f$ is Lipschitz. Clearly, every edge in $\Hyp(Z)$ (excluding the ones connected to $*$) is of the form $(\{n\}\times u) \to (\{n+\eps\}\times u')$ for some $n\geq 0$, $\eps \in \{0,1\}$ and $u,u' \in [0,1]^m$ with $\|u-u'\|\leq 2^{-n}$. Then by construction and the triangle inequality, $\|a_{n,u}-a_{n+\eps,u'}\| \leq 3\cdot 2^{-n}$. Hence, by Lemma~\ref{lem:BBS},
$$d_{\Hyp_{\rm BBS}(Z)}(f(\{n\}\times u),f(\{n+\eps\}\times u')) = d_{\Hyp_{\rm BBS}(Z)}((n,a_{n,u}),(n+\eps,a_{n+\eps,u'})) \leq C'$$
for some universal constant $C'<\infty$. This proves that $f$ is Lipschitz.

Now we define a map $g: \Hyp_{\rm BBS}(Z) \to \Hyp(Z)$. Every point in $\Hyp_{\rm BBS}(Z)$ is of the form $(n,a)$ where $n\geq 0$ and $a \in A_n$. Then $a \in Q$ for some $Q \in \D_n(Z)$. Let $u_{n,a} \in Q^{(0)}$ be any point. Define
$$g(n,a) := (\{n\}\times u_{n,a}) \in \Hyp(Z).$$
Let us prove that $g$ is Lipschitz. Clearly, every edge in $\Hyp_{\rm BBS}(Z)$ is of the form $(n,a) \sim (n+\eps,a')$ for some $n\geq 0$, $\eps \in \{0,1\}$ and $a,a' \in Z$ with $\|a-a'\|\leq 2^{-n+2}$. Then by construction and the triangle inequality, $\|u_{n,a}-u_{n+\eps,a'}\| \leq 3\cdot 2^{-n+1}$. Hence, by Lemma~\ref{lem:Hyp(Z)distance},
$$d_{\Hyp(Z)}(g(n,a),g(n+\eps,a')) = d_{\Hyp(Z)}(\{n\}\times u_{n,a},\{n+\eps\}\times u_{n+\eps, a'}) \leq C'$$
for some universal constant $C'<\infty$. This proves that $g$ is Lipschitz.

Finally, we will see that $f$ and $g$ are coarsely inverse. Let $(n,a) \in \Hyp_{\rm BBS}(Z)$ where $n\geq 0$ and $a \in A_n$. It is clear from the construction and triangle inequality that $\|a_{n,u_{n,a}} - a\| \leq 2^{-n+1}$, and therefore Lemma~\ref{lem:BBS} implies
$$d_{\Hyp_{\rm BBS}(Z)}(f(g(n,a)),(n,a)) = d_{\Hyp_{\rm BBS}(Z)}((n,a_{n,u_{n,a}}),(n,a)) \leq C'$$
for some universal constant $C'<\infty$. Hence, $f\circ g$ has bounded uniform distance to the identity. The same argument but using Lemma~\ref{lem:Hyp(Z)distance} in place of Lemma~\ref{lem:BBS} reveals that $g\circ f$ has bounded uniform distance to the identity.
\end{proof}

\begin{corollary} \label{cor:BBS}
Let $m\in\N$ and $Z \sbs [0,1]^m$ be nonempty and compact with $\diam(Z)<1$. Then $\Hyp(Z)$ is visual and hyperbolic, and $\partial\Hyp(Z)$ is power quasisymmetrically equivalent to $Z$.
\end{corollary}

\begin{proof}
That $\Hyp(Z)$ is visual and hyperbolic follows from the following facts: $\Hyp(Z)$ is quasi-isometric to $\Hyp_{\rm BBS}(Z)$ (Proposition~\ref{prop:EuclideanqiBBS}), the filling $\Hyp_{\rm BBS}(Z)$ is visual and hyperbolic (Proposition~\ref{prop:BBS}), the property of being a visual hyperbolic space is preserved under quasi-isometric equivalence for geodesic metric spaces (follows from stability of geodesics \cite[Theorem~1.3.2]{Buyalo} and \cite[Corollary~1.3.5]{Buyalo}). That $\partial\Hyp(Z)$ is power quasisymmetrically equivalent to $Z$ follows from the following facts: $\Hyp(Z)$ and $\Hyp_{\rm BBS}(Z)$ are quasi-isometrically equivalent (Proposition~\ref{prop:EuclideanqiBBS}), quasi-isometrically equivalent visual hyperbolic spaces have power quasisymmetrically equivalent boundaries (Theorem~\ref{thm:BS}), and $\partial\Hyp_{\rm BBS}(Z)$ is power quasisymmetrically equivalent to $Z$ (Proposition~\ref{prop:BBS}).
\end{proof}

\begin{theorem} \label{thm:qifillingA}
Let $\Gamma$ be a finitely generated hyperbolic group. Let $m\in\N$ and $Z \sbs [0,1]^m$ such that $\partial\Gamma$ is power quasisymmetrically equivalent to $Z$. Then $\Gamma$ is quasi-isometric to the hyperbolic filling $\Hyp(Z)$ of Definition~\ref{def:subEuclidean}.
\end{theorem}

\begin{proof}
By rescaling, we may assume that $\diam(Z)<1$. The spaces $\partial\Hyp(Z)$ and $Z$ are power quasisymmetrically equivalent by Corollary~\ref{cor:BBS}. Hence, by assumption, $\partial\Hyp(Z)$ and $\partial\Gamma$ are power quasisymmetrically equivalent. The space $\Gamma$ is hyperbolic by assumption and is visual by \cite[Lemma~3.1]{Bestvina}, and the space $\Hyp(Z)$ is visual and hyperbolic by Corollary~\ref{cor:BBS}. Therefore, by Theorem~\ref{thm:BS}, $\Gamma$ and $\Hyp(Z)$ are quasi-isometric.
\end{proof}

\end{document}